\documentclass[11pt,letterpaper]{amsart}
\usepackage{amsmath,amssymb,amsfonts,amsthm,mathrsfs,mathtools}
\usepackage[margin=1.15in]{geometry}
\usepackage{hyperref}
\usepackage{enumitem}
\usepackage{microtype}
\usepackage{tikz-cd}
\usetikzlibrary{arrows.meta,calc}
\usetikzlibrary{babel}
\usepackage{mathrsfs}
\usepackage{graphicx}

\newcommand{\Q}{\mathbb{Q}}
\newcommand{\Z}{\mathbb{Z}}
\newcommand{\R}{\mathbb{R}}
\newcommand{\C}{\mathbb{C}}

\newcommand{\cO}{\mathcal{O}}
\newcommand{\GL}{\mathrm{GL}}
\newcommand{\GO}{\mathrm{GO}}
\newcommand{\Tr}{\mathrm{Tr}}
\newcommand{\Sym}{\mathrm{Sym}}

\newcommand{\sh}{\mathrm{sh}}

\newcommand{\diag}{\mathrm{diag}}

\newcommand{\disc}{\mathrm{disc}}

\theoremstyle{plain}
\newtheorem{theorem}{Theorem}[section]
\newtheorem{proposition}[theorem]{Proposition}
\newtheorem{lemma}[theorem]{Lemma}
\newtheorem{corollary}[theorem]{Corollary}

\theoremstyle{definition}
\newtheorem{definition}[theorem]{Definition}
\newtheorem{remark}[theorem]{Remark}
\newtheorem{example}[theorem]{Example}

\tikzcdset{
  paperdiag/.style={
    row sep={4.2em,between origins},
    column sep={5.8em,between origins},
    cells={nodes={
      align=center,
      inner xsep=2pt,
      inner ysep=2pt,
      text height=1.7ex,
      text depth=.25ex
    }},
    every label/.append style={font=\scriptsize, inner sep=1pt}
  },
  paperdiag compact/.style={
    row sep={3.3em,between origins},
    column sep={4.6em,between origins},
    cells={nodes={
      align=center,
      inner xsep=1.5pt,
      inner ysep=1.5pt
    }},
    every label/.append style={font=\scriptsize, inner sep=1pt}
  }
}

\title{Minkowski shapes of pure number fields}
\author{Khai-Hoan Nguyen-Dang}
\date{\today}

\address{Morningside Center of Mathematics, Chinese Academy of Sciences, No.\ 55, Zhongguancun East Road, Beijing 100190, China}
\email{khaihoann@gmail.com}



\begin{document}
\maketitle
\begin{abstract}
We study the Minkowski shape of pure number fields
\[
K_a=\Q(\theta),\qquad \theta^n=a.
\]
For admissible parameters
satisfying an explicit local hypothesis at the primes dividing \(n\), we prove a
discrete--archimedean factorization
\[
\sh(K_a)=\bigl[C(a)^{\mathsf T}\diag\bigl(s_1(a),\dots,s_{n-1}(a)\bigr)C(a)\bigr],
\]
where the \(s_m(a)\) arise from normalized monomials and
\(C(a)\in\GL_{n-1}(\Q)\) comes from a normalized integral basis. This yields a
uniform odd/even rigidity dichotomy: for every odd \(n\ge 3\), the Minkowski
shape is a complete invariant among admissible pure degree-\(n\) fields, whereas
for \(n=2r\) it determines the core field \(\Q(|a|^{1/r})\); on the squarefree
admissible subfamily it is complete up to sign, although infinitely many
non-isomorphic pairs \(K_a\) and \(K_{-a}\) have the same shape. We also derive
explicit formulas for \(|\disc(K_a)|\), including exponent-vector and
divisor-lattice factorizations. Finally, we show that the pure-field shape locus
is supported on rational diagonal leaves in shape space: unconditionally it lies
in a countable union of closed leaves, while under the same local hypothesis only
finitely many leaves occur in each fixed degree. On a fixed normalized stratum,
the shape depends only on ratio variables, whereas discriminant growth is
governed by independent product variables.
\end{abstract}

\section{Introduction}

\subsection{Shapes and the pure family}

For a number field \(K/\Q\) of degree \(n\), the discriminant records the
covolume of the Minkowski lattice and the ramified primes.  The Minkowski shape
is the finer archimedean invariant obtained from the trace-zero lattice up to integral basis change, orthogonal similitude, and positive scale; thus it
is a point of
\[
\mathscr S_{n-1}
=
\GL_{n-1}(\Z)\backslash \GL_{n-1}(\R)/\GO_{n-1}(\R).
\]
It lies naturally between the discriminant and more rigid quadratic-form
invariants such as the integral trace form.

Two complementary pictures are already visible in the literature.  For generic
\(S_n\)-families of degrees \(3,4,5\), shapes equidistribute in the full
ambient space \(\mathscr S_{n-1}\)
\cite{Terr1997,BhargavaHarron2016,Hough2019}.  In thin families, by contrast,
the support can collapse onto much smaller loci; this happens for cyclic
families, pure cubic fields, pure prime-degree fields, Galois quartic fields,
and recent pure quartic and pure sextic families
\cite{MantillaSolerMonsurro2016,BolanosMantillaSoler2021,BolanosMantillaSoler2023,Harron2017PureCubic,Holmes2025PurePrime,PiperHHarron2020,DasKalaMukhopadhyayRay2025,JakharKalwaniyaRayRoy2026}.

This paper studies the pure family
\[
K_a=\Q(\theta),\qquad \theta^n=a,
\]
where \(a\) is \emph{admissible}: it is \(n\)th-power-free, distinct from
\(\pm1\), and \(x^n-a\) is irreducible over \(\Q\).  Pure fields are explicit
enough to admit concrete integral bases and closed archimedean Gram matrices,
but rigid enough that their shapes occupy a very small part of the ambient
moduli space. 

The arithmetic input is the explicit integral-basis theorem of
Jakhar--Khanduja--Sangwan, in the normalized form used by the author \cite{NguyenDang2025MinPeriod}.  Under
Hypothesis~\textup{(H)}, namely, for every prime \(p\mid n\), either
\(v_p(a)=0\) or \(p\nmid v_p(a)\), the ring of integers \(\cO_{K_a}\) admits a
basis
\[
\Bigl\{1,\ \omega_m=\frac{\theta^m+\beta_m}{C_m(a)D_m(a)}
\ (1\le m\le n-1)\Bigr\},
\]
with \(D_m(a)\) supported on primes dividing \(n\) and with residue data
periodic in \(a\) modulo
\[
M(n)=n\,\operatorname{rad}(n)
\]
\cite{JakharKhandujaSangwan2021,NguyenDang2025MinPeriod}.  This produces a
finite-place datum \(S(a)\).  The actual shape, however, depends on finer
rational basis-change data, and the first point of the paper is to separate
these levels.

For \(C\in \GL_{n-1}(\Q)\), write
\[
\mathcal T_C
:=
\{[C^{\mathsf T}DC]: D>0\text{ diagonal}\}\subset \mathscr S_{n-1}.
\]
These rational diagonal leaves are the basic support pieces of the pure family.

\subsection{Discrete--archimedean factorization}

The main structural theorem is that, under Hypothesis~\textup{(H)}, the shape
splits into an explicit diagonal archimedean part and a discrete rational part.

\begin{theorem}[Discrete--archimedean factorization]
\label{thm:intro-factorization}
Let \(a\) be admissible and assume Hypothesis~\textup{(H)}.  Write
\[
a=\varepsilon\prod_{j=1}^{n-1} a_j^{\,j}
\]
for the strong decomposition, with \(\varepsilon\in\{\pm1\}\), the \(a_j\)
positive squarefree, and the \(a_j\) pairwise coprime.  Define
\[
C_m(a):=\prod_{j=1}^{n-1} a_j^{\lfloor jm/n\rfloor},
\qquad
e_m:=\frac{\theta^m}{C_m(a)},
\qquad
s_m(a):=\frac{|a|^{2m/n}}{C_m(a)^2}
\qquad
(1\le m\le n-1).
\]
Then:
\begin{enumerate}[label=\textup{(\roman*)},leftmargin=2.8em]
\item
The normalized monomials \(e_1,\dots,e_{n-1}\) are pairwise orthogonal for the
Minkowski inner product, satisfy \(\Tr_{K_a/\Q}(e_m)=0\), and give the diagonal
trace-zero Gram matrix
\[
\mathrm{Gram}(e_1^\perp,\dots,e_{n-1}^\perp)
=
n^3\diag\bigl(s_1(a),\dots,s_{n-1}(a)\bigr),
\qquad
e_m^\perp:=n e_m.
\]

\item
There exists an upper-triangular matrix
\[
C(a)\in \GL_{n-1}(\Q),
\]
arising from any normalized integral basis of \(\cO_{K_a}\), such that
\[
\sh(K_a)
=
\Bigl[
C(a)^{\mathsf T}\diag\bigl(s_1(a),\dots,s_{n-1}(a)\bigr)C(a)
\Bigr].
\]

\item
If
\[
U_{n-1}^+(\Z)
:=
\{U\in \GL_{n-1}(\Z): U \text{ is upper triangular with diagonal }1\},
\]
then the right coset of \(C(a)\) modulo \(U_{n-1}^+(\Z)\) is independent of the
chosen normalized integral basis.  Thus
\[
\Xi(a):=[C(a)]\in \GL_{n-1}(\Q)/U_{n-1}^+(\Z)
\]
is well-defined and
\[
\sh(K_a)
=
\Bigl[
\Xi(a)^{\mathsf T}\diag\bigl(s_1(a),\dots,s_{n-1}(a)\bigr)\Xi(a)
\Bigr].
\]
\end{enumerate}
\end{theorem}

The class \(\Xi(a)\) is a convenient normalized stratum parameter attached to
the chosen pure presentation.  It organizes the discrete rational part of the
factorization and the
supporting leaf is the object \(\mathcal T_C\).

\subsection{Arithmetic consequences}

The factorization theorem has three immediate arithmetic outputs: it identifies
the field recovered from diagonal ratios, it yields a sharp odd/even dichotomy
for shape as an isomorphism invariant, and it gives closed discriminant
formulas.

\begin{theorem}[Parity dichotomy for pure-field shapes]
\label{thm:intro-parity}
Let \(a,b\) be admissible.
\begin{enumerate}[label=\textup{(\roman*)},leftmargin=2.8em]
\item
If \(n\) is odd and
\[
\sh(K_a)=\sh(K_b),
\]
then
\[
K_a\simeq K_b.
\]

\item
If \(n=2r\) is even and
\[
\sh(K_a)=\sh(K_b),
\]
then
\[
\Q(|a|^{1/r})=\Q(|b|^{1/r}).
\]

\item
If \(n=2r\) is even, then there exist infinitely many positive squarefree
admissible integers \(a\) satisfying Hypothesis~\textup{(H)} such that
\[
\sh(K_a)=\sh(K_{-a})
\qquad\text{but}\qquad
K_a\not\simeq K_{-a}.
\]

\item
If \(n=2r\) is even and \(a,b\) are squarefree admissible, then
\[
\sh(K_a)=\sh(K_b)
\quad\Longrightarrow\quad
|a|=|b|,
\]
hence
\[
K_b\simeq K_a
\qquad\text{or}\qquad
K_b\simeq K_{-a}.
\]
\end{enumerate}
\end{theorem}

Thus odd degree and even degree behave differently.  In odd degree the diagonal
ratio field already recovers the field.  In even degree the shape determines
only the core field \(\Q(|a|^{1/r})\) in general, and sign is the only
remaining ambiguity on the squarefree admissible subfamily.

\begin{theorem}[Discriminant factorization]
\label{thm:intro-discriminant}
Assume Hypothesis~\textup{(H)}.  Then
\[
|\disc(K_a)|
=
n^n\left(\prod_{m=1}^{n-1}s_m(a)\right)\bigl(\det C(a)\bigr)^2.
\]
Equivalently, if
\[
a=\varepsilon\prod_{j=1}^{n-1} a_j^{\,j},
\]
then
\[
|\disc(K_a)|
=
\kappa_n\bigl(S(a)\bigr)\prod_{j=1}^{n-1} a_j^{\,n-\gcd(j,n)},
\]
where \(\kappa_n\bigl(S(a)\bigr)\) depends only on the periodic datum \(S(a)\).
If
\[
B_d(a):=
\prod_{\substack{1\le j\le n-1\\ d\mid j}} a_j
\qquad (d\mid n),
\]
then
\[
|\disc(K_a)|
=
\kappa_n\bigl(S(a)\bigr)\cdot B_1(a)^{\,n-1}
\cdot \prod_{\substack{d\mid n\\ 1<d<n}} B_d(a)^{-\varphi(d)}.
\]
\end{theorem}

This divisor-lattice factorization is the arithmetic shadow of the later
geometry: after the finite periodic factor has been removed, discriminant growth
is governed by squarefree divisor-products, while the archimedean position
inside a fixed leaf is controlled by independent ratio variables.

\subsection{Geometric consequences}

From a geometric point of view, the factorization forces the pure family onto a very thin subset of the ambient shape space. The next theorem records that support statement in the language of rational diagonal leaves.  Later, in paired coordinates, the variables attached to non-coprime indices are
packaged into a finite tuple of discrete parameters \(\delta(a)\).

\begin{theorem}[Geometric support of the pure family]
\label{thm:intro-geometry}
For \(C\in \GL_{n-1}(\Q)\), let
\[
\mathcal T_C
:=
\{[C^{\mathsf T}DC]: D>0\text{ diagonal}\}\subset \mathscr S_{n-1}.
\]
Then:
\begin{enumerate}[label=\textup{(\roman*)},leftmargin=2.8em]
\item
Every admissible pure-field shape lies on a rational diagonal leaf:
\[
\sh(K_a)\in \mathcal T_C
\qquad\text{for some }C\in \GL_{n-1}(\Q).
\]

\item
Each leaf \(\mathcal T_C\) is closed in \(\mathscr S_{n-1}\) and is the image
of a rational maximal flat orbifold of dimension \(n-2\).

\item
There exists a finite set
\[
\mathcal H_n\subset M_{n-1}(\Z)
\]
of full-rank right Hermite normal form matrices, depending only on \(n\), such
that
\[
\{\sh(K_a): a \text{ admissible and satisfying Hypothesis~\textup{(H)}}\}
\subset
\bigcup_{H\in \mathcal H_n}\mathcal T_H.
\]

\item
On a fixed normalized stratum \(\Xi\) and fixed discrete label \(\delta\), the
shape depends only on the ratio variables
\[
\rho_j=\frac{a_j}{a_{n-j}}
\qquad (\gcd(j,n)=1),
\]
whereas, after fixing the periodic discriminant factor, discriminant bounds
depend only on the product variables
\[
P_j=a_j a_{n-j}
\qquad (\gcd(j,n)=1).
\]
\end{enumerate}
\end{theorem}

In particular, the natural counting problem for pure-field shapes is not
ambient.  The continuous motion takes place inside rational diagonal leaves, and
on a fixed stratum the archimedean and discriminant variables separate into
ratio and product coordinates. Moreover, they form a model thin family for Minkowski-shape questions.  In a
generic family, the natural problem is ambient: one studies how shapes spread
through the full space \(\mathscr S_{n-1}\).  In a thin family, however, the
first problem is structural: one must identify the geometric support of the
family inside shape space, the finite-place arithmetic data that select the
relevant support pieces, and the intrinsic coordinates that govern motion on
each such piece.  The main point of this paper is that all of these layers can
be carried out uniformly for the pure family.
We prove that the shape factors into an explicit diagonal archimedean part and
a discrete rational part, that the admissible pure-field locus is supported on
rational diagonal leaves, and that on each fixed normalized stratum the shape
depends only on ratio variables whereas discriminant growth depends only on
independent product variables.  Thus the family provides a concrete model
in which one can see, in arbitrary degree, the mechanism that a general
thin-family shape theory should explain: finite-place data determine the
supporting leaf, archimedean data move the shape inside that leaf, and
discriminant bounds become transverse hyperbolic product constraints.

\subsection*{Organization of the paper}

The paper is organized as follows.
Section~\ref{sec:minkowski-shape} develops the trace-zero Minkowski shape and
its basic functorial properties.
Section~\ref{sec:gram-pure-fields} introduces admissible parameters, strong
decomposition, and the normalized monomials.
Section~\ref{sec:shapes-pure-fields} combines the monomial model with the
normalized integral-basis theorem under Hypothesis~\textup{(H)}, producing the
factorization formula and the normalized stratum class \(\Xi(a)\).
Section~\ref{sec:shape-complete} proves the arithmetic consequences in odd and
even degree.
Section~\ref{sec:disc-divisor} establishes the discriminant formulas.
Section~\ref{sec:paired-rational-leaves} studies rational diagonal leaves and
the fixed-stratum ratio/product mechanism. Finally, we record the three next problems left open by the present support theory in Subsection~\ref{sec:outlook-open-problems}.

\subsection*{Acknowledgements}
 We thank Morningside Center of Mathematics, Chinese Academy of Sciences, for its support and a stimulating research environment.

\section{Minkowski shape}\label{sec:minkowski-shape}

Fix an integer $n\ge 3$ and put $m=n-1$. All number fields are finite extensions of $\Q$.
Let $(W,\langle\cdot,\cdot\rangle)$ be a real inner product space of dimension $m$.
Its orthogonal similitude group is
\[
\GO(W):=\{g\in \GL(W): \exists \lambda(g)\in \R_{>0}\ \text{s.t.}\ \langle gx,gy\rangle=\lambda(g)\langle x,y\rangle\ \forall x,y\in W\}.
\]
After choosing an orthonormal basis we identify $\GO(W)$ with
\begin{equation}\label{eq:GOmR}
\GO_m(\R)=\{g\in \GL_m(\R): g^{\mathsf T}g=\lambda I_m \ \text{for some }\lambda\in \R_{>0}\}.
\end{equation}

Throughout, we use the \emph{row-vector convention}: a full-rank lattice $\Lambda\subset \R^m$ is written as
$\Lambda=\Z^m g$ with $g\in \GL_m(\R)$, and basis change acts on the left by $\GL_m(\Z)$.
The Euclidean structure on $\R^m$ acts on the right by $\GO_m(\R)$.

To speak about shape, one first needs the correct moduli space of lattices up to integral change of basis, Euclidean isometry, and overall scale. The following definition fixes that ambient space once and for all.

\begin{definition}[Shape space]\label{def:shape-space}
The \emph{shape space} of rank-$m$ lattices is
\[
\mathscr S_m:=\GL_m(\Z)\backslash \GL_m(\R)/\GO_m(\R).
\]
\end{definition}

Under the row convention $\Lambda=\Z^m g$, the map
\[
\GL_m(\R)\longrightarrow \Sym_m^+(\R),\qquad g\longmapsto gg^{\mathsf T}
\]
induces an isomorphism
\[
\GL_m(\R)/\GO_m(\R)\xrightarrow{\ \sim\ }\Sym_m^+(\R)/\R_{>0}.
\]
Hence
\[
\mathscr S_m \simeq \GL_m(\Z)\backslash \bigl(\Sym_m^+(\R)/\R_{>0}\bigr),
\]
where $\GL_m(\Z)$ acts by congruence $A\mapsto UAU^{\mathsf T}$.

Let $K/\Q$ be a number field of degree $n$ with ring of integers $\cO_K$.
Write $r_1$ for the number of real embeddings and $r_2$ for the number of complex conjugate pairs, so
$n=r_1+2r_2$. Let
\[
K_\R:=K\otimes_\Q \R \cong \R^{r_1}\times \C^{r_2}
\]
via the Minkowski embedding $J:K\hookrightarrow K_\R$. The Minkowski embedding carries a canonical positive-definite form coming from the trace. This is the Euclidean structure behind all later Gram-matrix computations.

\begin{definition}[Minkowski inner product]\label{def:trace-inner}
Define a positive-definite inner product on $K_\R$ by
\begin{equation}\label{eq:trace-inner-product}
\langle x,y\rangle_\infty:=\Tr_{K/\Q}(x\overline y)\qquad (x,y\in K_\R),
\end{equation}
where $\Tr_{K/\Q}$ is extended $\R$-linearly to $K_\R$.
\end{definition}

For $\alpha,\beta\in K$ one has
\[
\langle J(\alpha),J(\beta)\rangle_\infty
=\sum_{\sigma:K\hookrightarrow \C}\sigma(\alpha)\,\overline{\sigma(\beta)}
=\sum_{\sigma\text{ real}}\sigma(\alpha)\sigma(\beta)
  +2\sum_{\tau\text{ complex pairs}}\Re\bigl(\tau(\alpha)\overline{\tau(\beta)}\bigr).
\]

Define the trace-zero subspace
\[
K^0:=\{\alpha\in K:\Tr_{K/\Q}(\alpha)=0\},\qquad K_\R^0:=K^0\otimes_\Q \R\subset K_\R.
\]
Since $\langle x,J(1)\rangle_\infty=\Tr(x)$ and $\langle J(1),J(1)\rangle_\infty=\Tr(1)=n$,
we have an orthogonal decomposition
\[
K_\R=\R\cdot J(1)\oplus K_\R^0.
\]
Let $\pi:K_\R\to K_\R^0$ denote the orthogonal projection with respect to $\langle\cdot,\cdot\rangle_\infty$. Since the line spanned by \(1\) plays a distinguished role, the relevant lattice for shape is its trace-zero complement. The next definition packages that complement in an integral way.

\begin{definition}[Trace-zero lattice]\label{def:perp}
Let $K/\Q$ have degree $n$, and let $\mathcal L\subset \cO_K$ be a full-rank $\Z$-submodule containing $1$.
Define
\begin{equation}\label{eq:perp-def}
\mathcal L^\perp:=\{\alpha^\perp:=n\alpha-\Tr_{K/\Q}(\alpha)\cdot 1:\alpha\in \mathcal L\}\subset K^0.
\end{equation}
For an order $\cO\subset K$ we write $\cO^\perp$; for $\mathcal L=\cO_K$ we write $\cO_K^\perp$.
\end{definition}

With the trace-zero lattice in place, the shape of a number field is simply the shape of that lattice in the ambient Minkowski space.

\begin{definition}[Minkowski shape of a lattice in a number field]\label{def:minkowski-shape}
Let $K/\Q$ have degree $n$, and let $\mathcal L\subset \cO_K$ be a full-rank $\Z$-submodule containing $1$.
The \emph{Minkowski shape} of $\mathcal L$ is
\[
\operatorname{sh}(K,\mathcal L):=[J(\mathcal L^\perp)]\in \mathscr S_{n-1}.
\]
For an order $\cO\subset K$ we write $\operatorname{sh}(K,\cO)$, and for the maximal order we write
\[
\operatorname{sh}(K):=\operatorname{sh}(K,\cO_K).
\]
\end{definition}

The first point is that passing from \(\mathcal L\) to \(\mathcal L^\perp\) removes exactly one obvious direction, namely the copy of \(\Z\cdot 1\). The next lemma makes that precise and identifies \(\mathcal L^\perp\) as the correct rank-\((n-1)\) lattice.

\begin{lemma}[Exact sequence and rank]\label{lem:exact-sequence}
Let $K/\Q$ have degree $n$, and let $\mathcal L\subset \cO_K$ be a full-rank $\Z$-submodule containing $1$.
Then the map
\[
\mathcal L\longrightarrow \mathcal L^\perp,\qquad \alpha\longmapsto \alpha^\perp
\]
is a surjective homomorphism with kernel $\Z\cdot 1$.
Consequently, $\mathcal L^\perp$ is a free $\Z$-module of rank $n-1$ and
\[
0\longrightarrow \Z\cdot 1\longrightarrow \mathcal L\longrightarrow \mathcal L^\perp\longrightarrow 0
\]
is exact.
\end{lemma}

\begin{proof}
Surjectivity is tautological from the definition of \(\mathcal L^\perp\).
If \(\alpha^\perp=0\), then
\[
n\alpha=\Tr_{K/\Q}(\alpha)\cdot 1\in \Q\cdot 1,
\]
hence \(\alpha\in \Q\). Since \(\alpha\in \mathcal L\subset \cO_K\), we obtain
\[
\alpha\in \Q\cap \cO_K=\Z,
\]
so \(\alpha\in \Z\cdot 1\).
Conversely, for every \(m\in\Z\),
\[
(m\cdot 1)^\perp
=nm\cdot 1-\Tr_{K/\Q}(m\cdot 1)\cdot 1
=nm\cdot 1-nm\cdot 1=0.
\]
Thus
\[
\ker(\alpha\mapsto \alpha^\perp)=\Z\cdot 1,
\qquad
\mathcal L^\perp \simeq \mathcal L/(\Z\cdot 1).
\]

Since \(\mathcal L^\perp\subset K^0\) and \(K^0\) is a \(\Q\)-vector space,
\(\mathcal L^\perp\) is torsion-free. Being finitely generated, it is therefore a free
abelian group. Its rank is
\[
\mathrm{rank}_\Z(\mathcal L^\perp)
=\mathrm{rank}_\Z(\mathcal L)-\mathrm{rank}_\Z(\Z\cdot 1)=n-1.
\]
This proves the exactness and the rank statement.
\end{proof}

The \(\perp\)-construction is just the integral incarnation of orthogonal projection away from the \(1\)-direction. Making that identity explicit will let us compare our convention with the projection-based literature.

\begin{lemma}[Orthogonal projection model]\label{lem:projection}
Let $\pi:K_\R\to K_\R^0$ be the orthogonal projection onto $K_\R^0$ with respect to
$\langle\cdot,\cdot\rangle_\infty$.
Then for every full-rank $\Z$-submodule $\mathcal L\subset \cO_K$ containing $1$,
\[
\mathcal L^\perp = n\cdot \pi(\mathcal L)\subset K^0,
\qquad\text{and}\qquad
J(\mathcal L^\perp)=n\cdot \pi\bigl(J(\mathcal L)\bigr)\subset K_\R^0.
\]
\end{lemma}

\begin{proof}
For $\alpha\in K$, identifying $K$ with $J(K)\subset K_\R$, we have
\[
\pi(\alpha)=\alpha-\frac{\langle \alpha,1\rangle_\infty}{\langle 1,1\rangle_\infty}\cdot 1
=\alpha-\frac{\Tr(\alpha)}{n}\cdot 1.
\]
Multiplying by $n$ gives
\[
n\pi(\alpha)=n\alpha-\Tr(\alpha)=\alpha^\perp.
\]
This proves both assertions.
\end{proof}

\begin{remark}[Projection versus the \(\perp\)-model]
\label{rem:projection-versus-perp}
Let \(\pi:K_\R\to K_\R^0\) be the orthogonal projection.
By Lemma~\ref{lem:projection},
\[
\alpha^\perp=n\,\pi(\alpha)\qquad(\alpha\in K),
\]
and therefore
\[
J(\cO_K^\perp)=n\,\pi\bigl(J(\cO_K)\bigr).
\]
Thus the trace-zero lattice \(J(\cO_K^\perp)\) and the literal projected
lattice \(\pi(J(\cO_K))\) differ only by the global homothety \(n\), so they
define the same point of \(\mathscr S_{n-1}\).
In particular, this matches Holmes's convention: after describing the shape as
the projection of \(j(\cO_K)\) onto \(j(1)^\perp\), he immediately implements
it via the map \(\alpha\mapsto \alpha^\perp=n\alpha-\Tr(\alpha)\).
\end{remark}

Before specializing to pure fields, one should check that the whole construction depends only on the isomorphism class of the pair \((K,\mathcal L)\). The next proposition verifies exactly that functoriality.

\begin{proposition}[Isomorphism invariance]\label{prop:iso-invariance}
Let $f:K\to K'$ be a $\Q$-algebra isomorphism, and let $\mathcal L\subset \cO_K$ and $\mathcal L'\subset \cO_{K'}$
be full-rank $\Z$-submodules containing $1$ such that $f(\mathcal L)=\mathcal L'$.
Then
\[
\operatorname{sh}(K,\mathcal L)=\operatorname{sh}(K',\mathcal L').
\]
In particular, if $\cO\subset K$ is an order and $f(\cO)=\cO'\subset K'$ is the corresponding order, then
$\operatorname{sh}(K,\cO)=\operatorname{sh}(K',\cO')$.
\end{proposition}

\begin{proof}
Since $\Tr_{K'/\Q}(f(x))=\Tr_{K/\Q}(x)$ for every $x\in K$, one has
\[
f(\alpha^\perp)
=f\bigl(n\alpha-\Tr_{K/\Q}(\alpha)\cdot 1\bigr)
=nf(\alpha)-\Tr_{K'/\Q}(f(\alpha))\cdot 1
=f(\alpha)^\perp.
\]
Hence $f(\mathcal L^\perp)=\mathcal L'^\perp$.

Let $J:K\hookrightarrow K_\R$ and $J':K'\hookrightarrow K'_\R$ be the Minkowski embeddings, and let
$f_\R:K_\R\to K'_\R$ be the induced $\R$-linear isomorphism. For $x,y\in K$ we compute
\[
\langle f_\R(J(x)),f_\R(J(y))\rangle_{\infty,K'}
=\sum_{\sigma':K'\hookrightarrow \C}\sigma'(f(x))\,\overline{\sigma'(f(y))}
=\sum_{\sigma:K\hookrightarrow \C}\sigma(x)\,\overline{\sigma(y)}
=\langle J(x),J(y)\rangle_{\infty,K}.
\]
By $\R$-bilinearity, the same identity holds for all $x,y\in K_\R$, so $f_\R$ is an isometry
\[
(K_\R,\langle\cdot,\cdot\rangle_{\infty,K})\xrightarrow{\ \sim\ }(K'_\R,\langle\cdot,\cdot\rangle_{\infty,K'}).
\]
Since $f_\R\bigl(J(\mathcal L^\perp)\bigr)=J'(\mathcal L'^\perp)$, the lattices $J(\mathcal L^\perp)$ and
$J'(\mathcal L'^\perp)$ are isometric, hence represent the same point of $\mathscr S_{n-1}$.
\end{proof}

\emph{Ordered bases and right matrix action.}\label{conv:right-action}
Let $V$ be a real inner product space.
If $(v_1,\dots,v_m)$ is an ordered $m$-tuple of vectors in $V$ and $T\in M_m(\R)$, we write
\[
(v_1,\dots,v_m)T=(w_1,\dots,w_m)
\]
to mean
\[
w_j=\sum_{i=1}^m v_i\,T_{ij}\qquad(1\le j\le m).
\]
Equivalently, the $j$th \emph{column} of $T$ is the coordinate column of $w_j$ in the ordered basis
$(v_i)$.
With this convention, $w_j\in \mathrm{Span}_\R\{v_1,\dots,v_j\}$ for all $j$ if and only if $T$ is upper triangular.
Moreover, if $G(v)$ (resp.\ $G(w)$) denotes the Gram matrix in the $v$-basis (resp.\ $w$-basis), then
\[
G(w)=T^{\mathsf T}G(v)\,T.
\]

We will repeatedly move between ordered bases, so it is useful to record once the corresponding transformation rule for Gram matrices. This is the basic linear-algebra identity used throughout the paper.

\begin{lemma}[Gram matrices under basis change]\label{lem:gram-congruence}
Let $V$ be a real inner product space.
Let $v_1,\dots,v_d$ be a basis with Gram matrix $G=(\langle v_i,v_j\rangle)$.
If $w_i=\sum_j v_jT_{ji}$, i.e.
\[
(w_1,\dots,w_d)=(v_1,\dots,v_d)T
\qquad\text{with }T\in\GL_d(\R),
\]
then the Gram matrix in the $w$-basis is
\[
G_w=T^{\mathsf T}GT.
\]
\end{lemma}

\begin{proof}
Compute
\[
\langle w_i,w_j\rangle=\sum_{k,\ell}T_{ki}T_{\ell j}\langle v_k,v_\ell\rangle,
\]
which is exactly the $(i,j)$-entry of $T^{\mathsf T}GT$.
\end{proof}

\begin{lemma}[Minkowski Gram determinant equals absolute discriminant]\label{lem:gram-absdisc}
Let $K/\Q$ have degree $n$, and let $v=(v_1,\dots,v_n)$ be a $\Q$-basis of $K$.
Let $\{\sigma_i\}_{i=1}^n$ be the set of embeddings $K\hookrightarrow\C$, counting complex embeddings separately,
and form the embedding matrix
\[
M=(\sigma_i(v_j))_{i,j}.
\]
Let
\[
G_\infty(v):=(\langle J(v_i),J(v_j)\rangle_\infty)_{i,j}
\]
be the Gram matrix in the Minkowski inner product \eqref{eq:trace-inner-product}.
Then
\[
\det G_\infty(v)=|\disc(v)|.
\]
In particular, if $v$ is an integral basis of $\cO_K$, then
\[
\det G_\infty(v)=|\disc(K)|.
\]
\end{lemma}

\begin{proof}
By definition,
\[
(G_\infty(v))_{j,k}
=\sum_{i=1}^n \sigma_i(v_j)\,\overline{\sigma_i(v_k)}.
\]
Hence, if \(M=(\sigma_i(v_j))_{i,j}\), then
\[
G_\infty(v)=M^{\mathsf T}\overline{M}.
\]
Therefore
\[
\det G_\infty(v)
=\det\!\bigl(M^{\mathsf T}\overline{M}\bigr)
=\det(M)\,\overline{\det(M)}
=|\det(M)|^2.
\]
On the other hand, the usual discriminant of \(v\) is
\[
\disc(v)=\det(M)^2.
\]
Thus
\[
|\disc(v)|=|\det(M)|^2=\det G_\infty(v).
\]
\end{proof}

As a first concrete payoff, one can compute the covolume of the trace-zero lattice in closed form. This gives the precise scale relating the shape lattice to the discriminant of the field.

\begin{proposition}[Covolume of $J(\cO_K^\perp)$]\label{prop:covolume-tracezero}
Let $K/\Q$ have degree $n$ and let
\[
\cO_K^\perp=\{\,n\alpha-\Tr(\alpha)\cdot 1:\alpha\in \cO_K\,\}.
\]
be the trace-zero lattice of $\cO_K$.
Then the covolume of the Euclidean lattice
\[
J(\cO_K^\perp)\subset (K_\R^0,\langle\cdot,\cdot\rangle_\infty)
\]
satisfies
\[
\mathrm{covol}\bigl(J(\cO_K^\perp)\bigr)^2
=\det\mathrm{Gram}\bigl(J(b_1^\perp),\dots,J(b_{n-1}^\perp)\bigr)
= n^{2n-3}|\disc(K)|,
\]
and hence
\[
\mathrm{covol}\bigl(J(\cO_K^\perp)\bigr)=n^{n-\frac32}|\disc(K)|^{1/2}.
\]
\end{proposition}

\begin{proof}
Choose a \(\Z\)-basis \((1,\omega_1,\dots,\omega_{n-1})\) of \(\cO_K\); this is possible
because \(1\) is primitive in the free abelian group \(\cO_K\) (equivalently, if
\(1=m\alpha\) with \(m\ge 2\) and \(\alpha\in \cO_K\), then
\(\alpha=1/m\in \Q\cap \cO_K=\Z\), impossible).

Set
\[
b_i^\perp:=n\omega_i-\Tr(\omega_i)\qquad(1\le i\le n-1).
\]
By Lemma~\ref{lem:exact-sequence}, \((b_1^\perp,\dots,b_{n-1}^\perp)\) is a \(\Z\)-basis of
\(\cO_K^\perp\).

Let
\[
v=(1,\omega_1,\dots,\omega_{n-1}),\qquad
v'=(1,b_1^\perp,\dots,b_{n-1}^\perp).
\]
Let
\[
t:=\bigl(\Tr(\omega_1),\dots,\Tr(\omega_{n-1})\bigr)^{\mathsf T}.
\]
Then the change-of-basis matrix from \(v\) to \(v'\) is
\[
T=
\begin{pmatrix}
1 & -t^{\mathsf T}\\
0 & nI_{n-1}
\end{pmatrix},
\qquad
\det(T)=n^{n-1}.
\]
Hence, by Lemma~\ref{lem:gram-congruence},
\[
G_\infty(v')=T^{\mathsf T}G_\infty(v)T,
\qquad
\det G_\infty(v')=(\det T)^2\det G_\infty(v)=n^{2(n-1)}\det G_\infty(v).
\]
By Lemma~\ref{lem:gram-absdisc}, $\det G_\infty(v)=|\disc(K)|$, so
\[
\det G_\infty(v')=n^{2(n-1)}|\disc(K)|.
\]

Finally, since $\Tr(b_i^\perp)=0$, we have
\[
\langle J(1),J(b_i^\perp)\rangle_\infty=0\qquad(1\le i\le n-1),
\]
so $G_\infty(v')$ is block diagonal:
\[
G_\infty(v')=
\bigl(\langle J(1),J(1)\rangle_\infty\bigr)\oplus \mathrm{Gram}\bigl(J(b_1^\perp),\dots,J(b_{n-1}^\perp)\bigr)
=(n)\oplus G^\perp.
\]
Thus
\[
\det G_\infty(v')=n\det(G^\perp),
\]
and comparing with the previous formula yields
\[
\det(G^\perp)=n^{2n-3}|\disc(K)|.
\]
\end{proof}

\section{Gram matrices in pure fields}\label{sec:gram-pure-fields}

Fix $n\ge 3$. The purpose of this section is to isolate a basis in which the archimedean geometry of a pure field becomes completely transparent. The payoff will be a diagonal trace-zero Gram matrix whose entries can be read off directly from the strong decomposition of \(a\).

We now restrict to the pure family and isolate the parameters for which the extension really has degree \(n\). This is the class of radicands we will work with throughout.

\begin{definition}[Admissible parameter]\label{def:admissible}
An integer $a\in\Z$ is \emph{admissible (for degree $n$)} if:
\begin{enumerate}[label=(\arabic*),leftmargin=2.2em]
\item $a$ is $n$th-power-free, i.e.\ $v_\ell(a)\in\{0,1,\dots,n-1\}$ for all primes $\ell$;
\item $a\neq \pm1$;
\item $x^n-a$ is irreducible over $\Q$.
\end{enumerate}
For admissible $a$, set $K_a=\Q(\theta)$ with $\theta^n=a$.
\end{definition}

\begin{remark}
We use the irreducibility assumption only to ensure that $[K_a:\Q]=n$, so that the trace is the sum over exactly
$n$ embeddings.
\end{remark}

To keep track of how primes divide \(a\), it is convenient to separate them according to their exponent modulo \(n\). The following factorization is the bookkeeping device that will organize both the integral and archimedean sides of the story.

\begin{definition}[Strong decomposition]\label{def:strong-decomp}
Let $a\in\Z$ be nonzero and $n$th-power-free.
There exist unique $\varepsilon\in\{\pm1\}$ and unique positive integers $a_1,\dots,a_{n-1}$ such that
\begin{equation}\label{eq:strong-decomp}
a=\varepsilon\prod_{j=1}^{n-1} a_j^j,
\end{equation}
where each $a_j$ is squarefree and the $a_j$ are pairwise coprime.
\end{definition}

This decomposition is more than convenient notation: it is completely determined by the prime factorization of \(a\). The next lemma records that uniqueness explicitly.

\begin{lemma}[Uniqueness of the strong decomposition]\label{lem:strong-uniq}
The decomposition \eqref{eq:strong-decomp} is unique.
\end{lemma}

\begin{proof}
For each prime $\ell$, since $a$ is $n$th-power-free we have $v_\ell(a)\in\{0,1,\dots,n-1\}$.
Assign $\ell$ to the unique index $j=v_\ell(a)$ (if $j\neq 0$), i.e.\ declare $\ell\mid a_j$ and $\ell\nmid a_k$
for $k\neq j$.
This produces squarefree, pairwise coprime $a_j$ and recovers $a$.
Uniqueness follows because $v_\ell(a)$ determines exactly which $a_j$ contains $\ell$.
\end{proof}

For $0\le m\le n-1$ define
\begin{equation}\label{eq:Cm-def}
C_m(a):=\prod_{j=1}^{n-1} a_j^{\lfloor jm/n\rfloor}.
\end{equation}

The normalization \(C_m(a)\) is chosen so that the monomials become integral without destroying their simple archimedean shape. The next proposition is the basic integrality statement.

\begin{proposition}[Integrality of normalized monomials]\label{prop:integrality-em}
Let $K_a=\Q(\theta)$ with $\theta^n=a$ and $a$ admissible.
For each $m=0,1,\dots,n-1$, the element
\[
\frac{\theta^m}{C_m(a)}
\]
is an algebraic integer.
\end{proposition}

\begin{proof}
Let $\ell$ be a prime and set $j:=v_\ell(a)\in\{0,1,\dots,n-1\}$.
If $j=0$, then $\ell\nmid a_k$ for every $k$, and therefore
\[
v_\ell\bigl(C_m(a)\bigr)=0.
\]
If $1\le j\le n-1$, then by the strong decomposition one has $\ell\mid a_j$ and $\ell\nmid a_k$ for every
$k\neq j$, so
\[
v_\ell\bigl(C_m(a)\bigr)=\Bigl\lfloor\frac{jm}{n}\Bigr\rfloor.
\]
Thus in all cases,
\[
v_\ell\!\left(\frac{a^m}{C_m(a)^n}\right)
=jm-n\,v_\ell\bigl(C_m(a)\bigr)
=jm-n\Bigl\lfloor\frac{jm}{n}\Bigr\rfloor\ge 0.
\]
Hence $a^m/C_m(a)^n\in \Z$.
Since
\[
\left(\frac{\theta^m}{C_m(a)}\right)^n=\frac{a^m}{C_m(a)^n}\in \Z,
\]
the element $\theta^m/C_m(a)$ is a root of the monic polynomial
\[
X^n-\frac{a^m}{C_m(a)^n}\in \Z[X].
\]
Therefore $\theta^m/C_m(a)$ is an algebraic integer.
\end{proof}

For $1\le m\le n-1$ define the normalized monomials
\begin{equation}\label{eq:em-def}
e_m:=\frac{\theta^m}{C_m(a)}\in \cO_{K_a}.
\end{equation}

Fix a complex $n$th root $\theta_0$ of $a$.
Since $x^n-a$ is irreducible, the $n$ embeddings $\sigma_k:K_a\hookrightarrow\C$ are determined by
\[
\sigma_k(\theta)=\zeta_n^k\theta_0,\qquad k=0,1,\dots,n-1,
\]
where $\zeta_n=e^{2\pi i/n}$. The great advantage of the pure basis is that the archimedean geometry becomes almost Fourier-theoretic. In particular, distinct monomials decouple completely under the Minkowski inner product.

\begin{proposition}[Orthogonality of monomials]\label{prop:orthogonality}
Let $a$ be admissible and let $1\le i,j\le n-1$.
Then
\[
\langle \theta^i,\theta^j\rangle_\infty=0\quad\text{if }i\neq j,
\qquad
\langle \theta^m,\theta^m\rangle_\infty=n\,|a|^{2m/n}\quad(1\le m\le n-1).
\]
\end{proposition}

\begin{proof}
By Definition~\ref{def:trace-inner},
\[
\langle \theta^i,\theta^j\rangle_\infty
=\sum_{k=0}^{n-1}\sigma_k(\theta^i)\,\overline{\sigma_k(\theta^j)}
=\theta_0^i\overline{\theta_0^j}\sum_{k=0}^{n-1}\zeta_n^{k(i-j)}.
\]
The geometric sum $\sum_{k=0}^{n-1}\zeta_n^{k(i-j)}$ vanishes unless $n\mid(i-j)$.
Since $1\le i,j\le n-1$, this happens if and only if $i=j$.
If $i=j=m$, the sum equals $n$ and
\[
\theta_0^m\overline{\theta_0^m}=|\theta_0|^{2m}=|a|^{2m/n}.
\]
\end{proof}

Define
\begin{equation}\label{eq:sm-def}
s_m(a):=\frac{|a|^{2m/n}}{C_m(a)^2}\qquad(1\le m\le n-1).
\end{equation}

Once the monomials are normalized, the previous orthogonality immediately becomes a diagonal Gram matrix. At the same time, these monomials already lie in the trace-zero hyperplane.

\begin{corollary}[Diagonal Gram matrix and trace-zero]\label{cor:diag-gram}
The Gram matrix of $(e_1,\dots,e_{n-1})$ in $(K_{a,\R},\langle\cdot,\cdot\rangle_\infty)$ is
\begin{equation}\label{eq:diag-gram}
\mathrm{Gram}(e_1,\dots,e_{n-1})=n\,\diag\bigl(s_1(a),\dots,s_{n-1}(a)\bigr).
\end{equation}
Moreover,
\[
\Tr_{K_a/\Q}(e_m)=0\qquad(1\le m\le n-1).
\]
\end{corollary}

\begin{proof}
The Gram-matrix formula follows from Proposition~\ref{prop:orthogonality} and $e_m=\theta^m/C_m(a)$.
For the trace,
\[
\Tr(\theta^m)=\theta_0^m\sum_{k=0}^{n-1}\zeta_n^{km}=0
\qquad(1\le m\le n-1),
\]
hence $\Tr(e_m)=0$ as well.
\end{proof}

Define
\[
e_m^\perp:=n e_m-\Tr(e_m)=n e_m\qquad(1\le m\le n-1).
\]

Passing from \(e_m\) to their \(\perp\)-versions only rescales the basis, so the trace-zero Gram matrix stays diagonal as well. This is the diagonal model that will later be conjugated by the integral-basis matrix.

\begin{lemma}[Diagonal Gram matrix for $(e_m^\perp)$]\label{lem:G0perp}
Let $G_0^\perp(a)$ be the Gram matrix of $(e_1^\perp,\dots,e_{n-1}^\perp)$.
Then
\begin{equation}\label{eq:G0perp}
G_0^\perp(a)=n^3\,\diag\bigl(s_1(a),\dots,s_{n-1}(a)\bigr).
\end{equation}
\end{lemma}

\begin{proof}
By Corollary~\ref{cor:diag-gram}, the Gram matrix of $(e_1,\dots,e_{n-1})$ is
$n\diag(s_1(a),\dots,s_{n-1}(a))$.
Scaling the basis by $n$ scales the Gram matrix by $n^2$, hence
\[
G_0^\perp(a)=n^2\cdot n\diag(s_1(a),\dots,s_{n-1}(a))
=n^3\diag(s_1(a),\dots,s_{n-1}(a)).
\]
\end{proof}

In odd degree, the diagonal entries retain enough algebraic independence to recover the underlying field. This is the key archimedean input behind the later completeness theorem.

\begin{lemma}[$\Q$-linear independence of $s_1(a),\dots,s_{n-1}(a)$ for odd $n$]
\label{lem:odd-s-independence}
Assume that $n$ is odd.
Then the numbers
\[
s_1(a),\dots,s_{n-1}(a)
\]
are $\Q$-linearly independent.
\end{lemma}

\begin{proof}
Let $\beta_a\in \R$ be the unique real root of $x^n-a$.
Since $n$ is odd and $x^n-a$ is irreducible, the set
\[
\{1,\beta_a,\beta_a^2,\dots,\beta_a^{n-1}\}
\]
is a $\Q$-basis of $K_a=\Q(\beta_a)$, and for $1\le m\le n-1$ we have
\[
s_m(a)=\frac{\beta_a^{2m}}{C_m(a)^2}.
\]
Suppose that
\[
\sum_{m=1}^{n-1} c_m s_m(a)=0
\qquad(c_m\in \Q).
\]
Let
\[
D:=\prod_{m=1}^{n-1} C_m(a)^2\in \Z_{>0}.
\]
Multiplying by $D$ yields
\[
\sum_{m=1}^{n-1} d_m\beta_a^{2m}=0,
\qquad d_m:=c_m\frac{D}{C_m(a)^2}\in \Q.
\]
For each $m\in\{1,\dots,n-1\}$, write
\[
2m=\varepsilon_m n+r_m,
\qquad \varepsilon_m\in\{0,1\},\quad 1\le r_m\le n-1.
\]
Because $n$ is odd, one has $2m\not\equiv 0\pmod n$, so $r_m\neq 0$.
Because $\gcd(2,n)=1$, the residues $r_m$ are pairwise distinct as $m$ varies.
Using $\beta_a^n=a$, we obtain
\[
\beta_a^{2m}=a^{\varepsilon_m}\beta_a^{r_m}.
\]
Therefore
\[
0=\sum_{m=1}^{n-1} d_m a^{\varepsilon_m}\beta_a^{r_m},
\]
which is a $\Q$-linear relation among the distinct basis elements
$\beta_a,\beta_a^2,\dots,\beta_a^{n-1}$.
Hence every coefficient $d_m a^{\varepsilon_m}$ is zero; since $a\neq 0$, it follows that each $d_m=0$.
Therefore each $c_m=0$, proving the claimed $\Q$-linear independence.
\end{proof}

\section{Shapes of pure fields}\label{sec:shapes-pure-fields}

\subsection{Integral-basis data and minimal periodicity}
We first recall the explicit integral-basis theorem of Jakhar--Khanduja--Sangwan in the normalized form
used by the author. See \cite[Theorem~2.1 and Remark~2.4]{NguyenDang2025MinPeriod}; this is a reformulation
of \cite[Theorem~1.6]{JakharKhandujaSangwan2021} under Hypothesis~\textup{(H)}. We isolate that hypothesis once here so that it can be invoked cleanly later.

\begin{definition}[Hypothesis (H)]\label{def:H}
Let $a$ be $n$th-power-free.
We say that $(n,a)$ satisfies \emph{Hypothesis~\textup{(H)}} if for every prime $p\mid n$, either
$v_p(a)=0$, or $p\nmid v_p(a)$.
\end{definition}

Hypothesis~\textup{(H)} holds, for example, when $a$ is squarefree or $\gcd(a,n)=1$. Under that assumption, the arithmetic complexity of \(\cO_{K_a}\) can be concentrated into controlled denominator factors at primes dividing \(n\) and lower-order corrections in the power basis. The next theorem is the precise form of that structure that we will use throughout.

\begin{theorem}[Normalized integral basis for pure fields]\label{thm:ND-basis}
Assume that $a$ is admissible and satisfies Hypothesis~\textup{(H)}.
Then there exist elements $\beta_m\in \Z[\theta]$, each a $\Z$-linear combination of
$1,\theta,\dots,\theta^{m-1}$, and nonnegative integers $k_{p,m}$ (for primes $p\mid n$ and
$1\le m\le n-1$), such that
\begin{equation}\label{eq:int-basis}
\Bigl\{1,\ \omega_m:=\frac{\theta^m+\beta_m}{C_m(a)\,D_m(a)}\ (1\le m\le n-1)\Bigr\}
\quad\text{is an integral basis of }\cO_{K_a},
\end{equation}
where
\[
D_m(a):=\prod_{p\mid n}p^{k_{p,m}}.
\]
Moreover, for each prime $p\mid n$, the residue class
\[
\beta_m \bmod p^{k_{p,m}}
\]
is determined by the explicit construction of \cite[Theorem~2.1]{NguyenDang2025MinPeriod}, and every prime
divisor of every $D_m(a)$ divides $n$ \cite[Remark~2.4]{NguyenDang2025MinPeriod}.
\end{theorem}

\begin{definition}[Normalized integral bases attached to \(S(a)\)]
\label{def:normalized-integral-basis}
Assume that \(a\) is admissible and satisfies Hypothesis~\textup{(H)}.  We fix
the denominator exponents \(k_{p,m}\) and the residue classes
\[
[\beta_m]_{p^{k_{p,m}}}
\]
provided by the explicit construction in Theorem~\ref{thm:ND-basis}.  An
integral basis
\[
\Bigl\{
1,\ 
\omega_m=\frac{\theta^m+\widetilde\beta_m}{C_m(a)D_m(a)}
\ (1\le m\le n-1)
\Bigr\}
\]
is called a \emph{normalized integral basis attached to \(S(a)\)} if
\[
D_m(a)=\prod_{p\mid n}p^{k_{p,m}}
\]
with these fixed exponents, and each \(\widetilde\beta_m\) is a lift of the
prescribed residue class \([\beta_m]_{p^{k_{p,m}}}\) for every \(p\mid n\).
Thus all normalized integral bases attached to \(S(a)\) have the same
denominators \(D_m(a)\).
\end{definition}

\begin{remark}[Size of the \(p\)-denominator exponents]
\label{rem:kpm-bound}
If \(p^e\parallel n\), then
\[
0\le k_{p,m}\le e\qquad(1\le m\le n-1).
\]
Indeed, in the explicit local construction preceding
\cite[Theorem~1.6]{JakharKhandujaSangwan2021},
the exponent \(k_{p,m}\) is defined as the largest nonnegative integer
not exceeding a quantity bounded above by \(e=v_p(n)\).
\end{remark}

The normalized integral basis naturally breaks its finite-place information into local packets, one for each prime dividing \(n\). The next definition names those packets and the global datum they assemble into.

\begin{definition}
\label{def:ND-shape}
For a prime \(p\mid n\), define the \emph{local shape datum}
\[
S_p(a):=\Bigl(\bigl(k_{p,m}\bigr)_{1\le m\le n-1},\
\bigl([\beta_m]_{p^{k_{p,m}}}\bigr)_{1\le m\le n-1}\Bigr),
\]
where, for \(k_{p,m}>0\), \([\beta_m]_{p^{k_{p,m}}}\) denotes the class of \(\beta_m\) in the
free \((\Z/p^{k_{p,m}}\Z)\)-module with basis \(1,\theta,\dots,\theta^{m-1}\), and for
\(k_{p,m}=0\) this class is understood to be trivial.
Define the \emph{global shape datum}
\[
S(a):=(S_p(a))_{p\mid n}.
\]
\end{definition}

A striking feature of this finite-place data is that it does not depend on the full size of \(a\), but only on a short congruence class. We prove that this periodicity is in fact sharp.

\begin{theorem}[Minimal periodicity modulus]\label{thm:min-periodicity}
Assume that $a$ is admissible and satisfies Hypothesis~\textup{(H)}.
Let $p^e\parallel n$.
Then the local datum $S_p(a)$ is determined by $a\bmod p^{e+1}$, and this precision is optimal.
Consequently, the global datum $S(a)$ is determined by
\[
a\bmod M(n),
\qquad
M(n):=\prod_{p^e\parallel n}p^{e+1}=n\,\operatorname{rad}(n),
\]
and $M(n)$ is minimal among global periods.
\end{theorem}

\begin{proof}
This is exactly \cite[Theorem~3.1]{NguyenDang2025MinPeriod}.
\end{proof}

\begin{remark}[Presentation dependence]
\label{rem:presentation}
The invariants \(S(a)\), \(C(a)\), \(\Xi(a)\), and the diagonal parameters
\[
s_1(a),\dots,s_{n-1}(a)
\]
depend on the chosen pure presentation
\[
K=\Q(\theta),\qquad \theta^n=a,
\]
and on the fixed power basis \(\{1,\theta,\dots,\theta^{n-1}\}\).

A field may admit several pure presentations, and this already happens in
prime degree; for example
\[
\Q(\sqrt[3]{2})=\Q(\sqrt[3]{4}).
\]
Thus these auxiliary data need not agree across different presentations of
the same field.

By contrast, the Minkowski shape \(\operatorname{sh}(K)\) is a field
invariant by Proposition~\ref{prop:iso-invariance}.
In particular, \(\Xi(a)\) should be viewed as a presentation-dependent
discrete datum attached to the normalized monomial model, not as an intrinsic
invariant of the abstract field.
\end{remark}

\subsection{Integral shape and Minkowski shape}

Whenever an integer or rational number is written as an element of a field \(K\), we
mean its image under the natural embedding \(\Q\hookrightarrow K\); thus \(t\) denotes
\(t\cdot 1_K\). In particular,
\[
n\alpha-\Tr_{K/\Q}(\alpha)
\]
is shorthand for
\[
n\alpha-\Tr_{K/\Q}(\alpha)\cdot 1_K.
\]

Assume henceforth that \(a\) is admissible and satisfies Hypothesis~\textup{(H)}.
Fix once and for all a normalized integral basis
\[
\mathcal B_a=
\Bigl\{1,\ \omega_m=\frac{\theta^m+\beta_m}{C_m(a)\,D_m(a)}\ (1\le m\le n-1)\Bigr\}
\]
as in Theorem~\ref{thm:ND-basis}. For \(1\le m\le n-1\), define
\[
b_m^\perp:=n\omega_m-\Tr(\omega_m)\in \cO_{K_a}^\perp.
\]
By Lemma~\ref{lem:exact-sequence}, the images of \(\omega_1,\dots,\omega_{n-1}\) form a
\(\Z\)-basis of \(\cO_{K_a}/\Z\), hence \((b_1^\perp,\dots,b_{n-1}^\perp)\) is a \(\Z\)-basis
of \(\cO_{K_a}^\perp\). The normalized monomial lattice is not usually the full ring of integers, but its defect is concentrated at the primes dividing \(n\). Away from those primes, it is already maximal.

\begin{lemma}[The normalized monomial lattice is maximal away from primes dividing $n$]
\label{lem:monomial-max-away-n}
Assume that $a$ is admissible and satisfies Hypothesis~\textup{(H)}.
Let
\[
L_0(a):=\Z\langle 1,e_1,\dots,e_{n-1}\rangle\subset \cO_{K_a}.
\]
Then
\[
[\cO_{K_a}:L_0(a)]=\prod_{m=1}^{n-1}D_m(a),
\]
hence the index $[\cO_{K_a}:L_0(a)]$ is supported only on primes dividing $n$.
In particular, for every prime $q\nmid n$,
\[
L_0(a)\otimes \Z_q=\cO_{K_a}\otimes \Z_q,
\qquad\text{and}\qquad
L_0(a)^\perp\otimes \Z_q=\cO_{K_a}^\perp\otimes \Z_q.
\]
\end{lemma}

\begin{proof}
Write the integral basis of Theorem~\ref{thm:ND-basis} as
\[
\mathcal B_a=\Bigl(1,\ \omega_m=\frac{\theta^m+\beta_m}{C_m(a)\,D_m(a)}\ (1\le m\le n-1)\Bigr).
\]
Since each $\beta_m$ is a $\Z$-linear combination of $1,\theta,\dots,\theta^{m-1}$, each $\omega_m$ is a
$\Q$-linear combination of $1,e_1,\dots,e_m$.
Thus the change-of-basis matrix $T$ from the ordered $\Z$-basis $(1,e_1,\dots,e_{n-1})$ of $L_0(a)$
to the ordered $\Z$-basis $\mathcal B_a$ of $\cO_{K_a}$ is upper triangular, with diagonal entries
\[
1,\ D_1(a)^{-1},\ \dots,\ D_{n-1}(a)^{-1}.
\]
Hence
\[
\det(T)=\prod_{m=1}^{n-1}D_m(a)^{-1}
\]
and therefore
\[
[\cO_{K_a}:L_0(a)]=|\det(T)|^{-1}=\prod_{m=1}^{n-1}D_m(a).
\]
By Theorem~\ref{thm:ND-basis}, every prime divisor of each $D_m(a)$ divides $n$, so the index is supported
only on primes dividing $n$.
Therefore
\[
L_0(a)\otimes\Z_q=\cO_{K_a}\otimes\Z_q\qquad(q\nmid n).
\]

Applying Lemma~\ref{lem:exact-sequence} to \(L_0(a)\) and to \(\cO_{K_a}\), and then
tensoring with \(\Z_q\) (which is flat over \(\Z\)), we obtain exact sequences
\[
0\to \Z_q \xrightarrow{\,t\mapsto t\cdot 1\,} L_0(a)\otimes \Z_q
   \to L_0(a)^\perp\otimes \Z_q \to 0
\]
and
\[
0\to \Z_q \xrightarrow{\,t\mapsto t\cdot 1\,} \cO_{K_a}\otimes \Z_q
   \to \cO_{K_a}^\perp\otimes \Z_q \to 0.
\]
Since \(L_0(a)\otimes \Z_q=\cO_{K_a}\otimes \Z_q\), the middle terms are equal, and the
left maps have the same image \(\Z_q\cdot 1\). Therefore the quotients are equal:
\[
L_0(a)^\perp\otimes \Z_q=\cO_{K_a}^\perp\otimes \Z_q.
\]
\end{proof}

This means that the passage from the diagonal monomial model to the true integral basis is encoded by a single rational upper-triangular matrix. The next lemma records its basic structure and arithmetic integrality properties.

\begin{lemma}[Triangular transition matrix]\label{lem:C-matrix}
Assume that $a$ is admissible and satisfies Hypothesis~\textup{(H)}, and let
\[
\Bigl\{1,\omega_m=\frac{\theta^m+\beta_m}{C_m(a)\,D_m(a)}\ (1\le m\le n-1)\Bigr\}
\]
be the normalized integral basis of Theorem~\ref{thm:ND-basis}.
Put
\[
b_m^\perp:=n\omega_m-\Tr(\omega_m)\in \cO_{K_a}^\perp,
\qquad
e_m^\perp:=n e_m\in L_0(a)^\perp
\qquad(1\le m\le n-1).
\]
Then:

\begin{enumerate}[label=(\alph*)]
\item There is a unique matrix $C(a)\in \GL_{n-1}(\Q)$ such that
\begin{equation}\label{eq:b-e-C}
(b_1^\perp,\dots,b_{n-1}^\perp)=(e_1^\perp,\dots,e_{n-1}^\perp)\,C(a),
\end{equation}
and $C(a)$ is upper triangular.

\item The diagonal entries are
\[
(C(a))_{m,m}=D_m(a)^{-1}\qquad(1\le m\le n-1).
\]
In particular,
\[
\det C(a)=\prod_{m=1}^{n-1}D_m(a)^{-1}.
\]

\item One has
\[
C(a)\in \GL_{n-1}(\Z[1/n]).
\]
More precisely, for every prime $q\nmid n$,
\[
C(a)\in \GL_{n-1}(\Z_q).
\]
\end{enumerate}
\end{lemma}

\begin{proof}
Write
\[
\beta_m=\sum_{r=0}^{m-1} b_{r,m}\theta^r
\qquad (b_{r,m}\in\Z),
\]
and adopt the convention $e_0:=1$ and $C_0(a):=1$.
Then
\begin{equation}\label{eq:omega-expand}
\omega_m
=\frac{\theta^m+\beta_m}{C_m(a)D_m(a)}
=\frac{1}{D_m(a)}\,e_m
 +\sum_{r=0}^{m-1}\frac{b_{r,m}}{D_m(a)}\frac{C_r(a)}{C_m(a)}\,e_r.
\end{equation}
Applying $\alpha\mapsto \alpha^\perp=n\alpha-\Tr(\alpha)$ kills the constant term, since $e_0^\perp=0$ and
$\Tr(e_r)=0$ for $1\le r\le n-1$.
Thus \eqref{eq:omega-expand} yields
\begin{equation}\label{eq:b-expand}
b_m^\perp
=\frac{1}{D_m(a)}\,e_m^\perp
 +\sum_{r=1}^{m-1}\frac{b_{r,m}}{D_m(a)}\frac{C_r(a)}{C_m(a)}\,e_r^\perp.
\end{equation}
This proves (a) and (b).

For (c), fix a prime \(q\nmid n\).
By Lemma~\ref{lem:monomial-max-away-n},
\[
L_0(a)^\perp\otimes \Z_q=\cO_{K_a}^\perp\otimes \Z_q.
\]
Therefore both \((e_1^\perp,\dots,e_{n-1}^\perp)\) and \((b_1^\perp,\dots,b_{n-1}^\perp)\)
are \(\Z_q\)-bases of the same \(\Z_q\)-lattice, so the transition matrix \(C(a)\) belongs to
\(\GL_{n-1}(\Z_q)\).

Since every entry of \(C(a)\) lies in \(\Q\), this implies
\[
C(a)\in M_{n-1}\!\Bigl(\Q\cap \bigcap_{q\nmid n}\Z_q\Bigr)
      = M_{n-1}(\Z[1/n]).
\]
Moreover,
\[
\det C(a)=\prod_{m=1}^{n-1}D_m(a)^{-1},
\]
and every prime divisor of every \(D_m(a)\) divides \(n\). Hence \(\det C(a)\) is a unit
of \(\Z[1/n]\), so in fact
\[
C(a)\in \GL_{n-1}(\Z[1/n]).
\]
\end{proof}

We now reach the central structural point of the paper: the actual Minkowski shape is obtained by conjugating an explicit diagonal archimedean matrix by the rational transition matrix coming from the integral basis. In other words, the archimedean and finite parts separate cleanly but do not decouple trivially.

\begin{theorem}[Discrete--archimedean factorization of the Minkowski shape]\label{thm:factorization}
Assume that $a$ is admissible and satisfies Hypothesis~\textup{(H)}.
Let $G^\perp(a)$ be the Gram matrix of the $\Z$-basis
\[
(J(b_1^\perp),\dots,J(b_{n-1}^\perp))
\]
of $J(\cO_{K_a}^\perp)$.
Then
\begin{equation}\label{eq:main-congruence}
G^\perp(a)=C(a)^{\mathsf T}\,G_0^\perp(a)\,C(a),
\end{equation}
where $C(a)$ is the transition matrix from Lemma~\ref{lem:C-matrix} and $G_0^\perp(a)$ is the diagonal matrix
from Lemma~\ref{lem:G0perp}.
Consequently, in shape space,
\begin{equation}\label{eq:shape-formula}
\operatorname{sh}(K_a)=\Bigl[\,C(a)^{\mathsf T}\,\diag\bigl(s_1(a),\dots,s_{n-1}(a)\bigr)\,C(a)\,\Bigr]\in \mathscr S_{n-1}.
\end{equation}
\end{theorem}

\begin{proof}
Equation \eqref{eq:b-e-C} gives the change of basis from $(e_1^\perp,\dots,e_{n-1}^\perp)$ to
$(b_1^\perp,\dots,b_{n-1}^\perp)$ with transition matrix $C(a)$.
Applying Lemma~\ref{lem:gram-congruence} gives \eqref{eq:main-congruence}.
Substituting \eqref{eq:G0perp} yields an overall scalar factor $n^3$, which is irrelevant in shape space,
because $\mathscr S_{n-1}$ is defined modulo positive homotheties.
\end{proof}

The normalized integral basis of Theorem~\ref{thm:ND-basis} is not canonical:
the reduced corrections are well determined modulo the indicated \(p\)-power
denominators, but different lifts may be chosen in the explicit construction;
see \cite[Remark~1.7]{JakharKhandujaSangwan2021} and \cite[Theorem~2.1]{NguyenDang2025MinPeriod}.
The next lemma shows that the unipotent right coset defined below is
independent of this choice.

\begin{lemma}[Well-definedness of the unipotent class]
\label{lem:Xi-well-defined}
Assume that \(a\) is admissible and satisfies Hypothesis~\textup{(H)}.
Let
\[
\mathcal B=
\Bigl\{1,\ \omega_m=\frac{\theta^m+\beta_m}{C_m(a)D_m(a)}\ (1\le m\le n-1)\Bigr\}
\]
and
\[
\mathcal B'=
\Bigl\{1,\ \omega_m'=\frac{\theta^m+\beta_m'}{C_m(a)D_m(a)}\ (1\le m\le n-1)\Bigr\}
\]
be two integral bases of \(\cO_{K_a}\) of the normalized form of
Theorem~\ref{thm:ND-basis}, with each \(\beta_m,\beta_m'\) a
\(\Z\)-linear combination of \(1,\theta,\dots,\theta^{m-1}\).
For \(1\le m\le n-1\), define
\[
b_m^\perp:=n\omega_m-\Tr_{K_a/\Q}(\omega_m),
\qquad
b_m'{}^\perp:=n\omega_m'-\Tr_{K_a/\Q}(\omega_m'),
\]
and let \(C(a),C'(a)\in \GL_{n-1}(\Q)\) be the unique matrices such that
\[
(b_1^\perp,\dots,b_{n-1}^\perp)
=(e_1^\perp,\dots,e_{n-1}^\perp)\,C(a),
\]
\[
(b_1'{}^\perp,\dots,b_{n-1}'{}^\perp)
=(e_1^\perp,\dots,e_{n-1}^\perp)\,C'(a).
\]
Then there exists \(U\in U_{n-1}^+(\Z)\) such that
\[
C'(a)=C(a)\,U,
\]
here
\[
U_{n-1}^+(\Z):=\{U\in \GL_{n-1}(\Z): U \text{ is upper triangular with all diagonal entries }1\}
\]
denotes the upper-unipotent subgroup; similarly for $U_{n-1}^+(\Z_p)$. In particular, the right coset \(C(a)\,U_{n-1}^+(\Z)\) is independent of the
choice of normalized integral basis.
\end{lemma}

\begin{proof}
Let \(T,T'\in \GL_n(\Q)\) be the transition matrices from the ordered basis
\[
(1,e_1,\dots,e_{n-1})
\]
to the ordered bases
\[
(1,\omega_1,\dots,\omega_{n-1})
\qquad\text{and}\qquad
(1,\omega_1',\dots,\omega_{n-1}').
\]
Because each \(\beta_m\) and \(\beta_m'\) is a \(\Z\)-linear combination of
\(1,\theta,\dots,\theta^{m-1}\), both \(T\) and \(T'\) are upper triangular.
Their diagonal entries are
\[
1,\ D_1(a)^{-1},\ \dots,\ D_{n-1}(a)^{-1},
\]
so \(T\) and \(T'\) have the same diagonal.

Since both ordered tuples are \(\Z\)-bases of the same lattice \(\cO_{K_a}\),
the matrix
\[
V:=T^{-1}T'
\]
belongs to \(\GL_n(\Z)\).
Because \(T\) and \(T'\) are upper triangular with the same diagonal, \(V\) is
upper triangular with all diagonal entries equal to \(1\).
Hence \(V\) has a block decomposition
\[
V=
\begin{pmatrix}
1 & * \\
0 & U
\end{pmatrix}
\qquad\text{with }U\in U_{n-1}^+(\Z).
\]

Now
\[
(1,\omega_1',\dots,\omega_{n-1}')
=
(1,\omega_1,\dots,\omega_{n-1})\,V.
\]
Apply the \(\Q\)-linear map \(\alpha\mapsto \alpha^\perp=n\alpha-\Tr(\alpha)\).
Since \(1^\perp=0\), this gives
\[
(b_1'{}^\perp,\dots,b_{n-1}'{}^\perp)
=
(b_1^\perp,\dots,b_{n-1}^\perp)\,U.
\]
Using
\[
(b_1^\perp,\dots,b_{n-1}^\perp)
=(e_1^\perp,\dots,e_{n-1}^\perp)\,C(a)
\]
and
\[
(b_1'{}^\perp,\dots,b_{n-1}'{}^\perp)
=(e_1^\perp,\dots,e_{n-1}^\perp)\,C'(a),
\]
we obtain
\[
(e_1^\perp,\dots,e_{n-1}^\perp)\,C'(a)
=
(e_1^\perp,\dots,e_{n-1}^\perp)\,C(a)\,U.
\]
Since \(e_1^\perp,\dots,e_{n-1}^\perp\) is a basis of \(K_a^0\), it follows that
\[
C'(a)=C(a)\,U.
\]
\end{proof}

The matrix \(C(a)\) itself depends on the chosen normalized integral basis, but Lemma~\ref{lem:Xi-well-defined} shows that a smaller quotient survives canonically. This is the discrete invariant that records the rational part of the shape factorization.

\begin{definition}[Unipotent discrete invariant attached to a normalized presentation]
\label{def:Xi}
Choose any normalized integral basis attached to \(S(a)\), in the sense of
Definition~\ref{def:normalized-integral-basis},
\[
\Bigl\{1,\ \omega_m=\frac{\theta^m+\beta_m}{C_m(a)D_m(a)}
\ (1\le m\le n-1)\Bigr\},
\]
and let \(C(a)\) be the matrix of Lemma~\ref{lem:C-matrix}.
By Lemma~\ref{lem:Xi-well-defined}, the right coset of \(C(a)\) modulo
\(U_{n-1}^+(\Z)\) is independent of the chosen normalized integral basis.
We define
\[
\Xi(a):=[C(a)]\in \GL_{n-1}(\Q)\big/U_{n-1}^+(\Z).
\]

For a right coset \(\Xi=C\,U_{n-1}^+(\Z)\in \GL_{n-1}(\Q)/U_{n-1}^+(\Z)\) and a positive
diagonal matrix \(D\), define
\[
[\Xi^{\mathsf T}D\Xi]:=[C^{\mathsf T}DC]\in \mathscr S_{n-1}.
\]
This is well-defined: if \(C'=CU\) with \(U\in U_{n-1}^+(\Z)\), then
\[
C'^{\mathsf T}DC' = U^{\mathsf T}(C^{\mathsf T}DC)U,
\]
and \(U\in \GL_{n-1}(\Z)\), so \(C'^{\mathsf T}DC'\) and \(C^{\mathsf T}DC\) represent the
same point of \(\mathscr S_{n-1}\).
\end{definition}

After quotienting out the basis ambiguity, the factorization becomes conceptually cleaner: the shape is determined by the diagonal archimedean data together with the unipotent class \(\Xi(a)\).

\begin{corollary}[Factorization through the unipotent class]
\label{cor:Xi-factorization}
Assume that \(a\) is admissible and satisfies Hypothesis~\textup{(H)}.
Then for any representative \(C\in \Xi(a)\),
\[
\operatorname{sh}(K_a)
=
\Bigl[\,C^{\mathsf T}\,\diag\bigl(s_1(a),\dots,s_{n-1}(a)\bigr)\,C\,\Bigr].
\]
Equivalently, with the notation of Definition~\ref{def:Xi},
\[
\operatorname{sh}(K_a)
=
\Bigl[\Xi(a)^{\mathsf T}\,\diag\bigl(s_1(a),\dots,s_{n-1}(a)\bigr)\,\Xi(a)\Bigr].
\]
\end{corollary}

\begin{proof}
Immediate from Theorem~\ref{thm:factorization} and
Definition~\ref{def:Xi}.
\end{proof}

Although \(C(a)\) varies with \(a\), its denominators are uniformly controlled in terms of the degree alone. This quantitative bound will later let us pass from rational matrices to finitely many support leaves.

\begin{proposition}[A sharper uniform common denominator for \(C(a)\)]
\label{prop:uniform-denominator}
Fix \(n\ge 3\) and let $a$ be admissible and satisfy Hypothesis~\textup{(H)}.
For
\[
N_n^{\sharp}:=\prod_{p^e\parallel n} p^{\,e+n-2}
\]
and every admissible \(a\) satisfying Hypothesis~\textup{(H)}, one has
\[
N_n^{\sharp}\,C(a)\in M_{n-1}(\Z).
\]
Equivalently, all denominators of the entries of \(C(a)\) are bounded in terms
of \(n\) alone, and the exponent \(e+n-2\) suffices at each prime
\(p^e\parallel n\).
\end{proposition}

\begin{proof}
By Lemma~\ref{lem:C-matrix}\textup{(c)}, we already know that
\[
C(a)\in \GL_{n-1}(\Z[1/n]).
\]
Thus only primes \(p\mid n\) can occur in denominators of the entries of \(C(a)\). It
therefore suffices to bound, for each fixed \(p^e\parallel n\), the negative \(p\)-adic
valuations of the entries of \(C(a)\).

By Remark~\ref{rem:kpm-bound},
\[
0\le k_{p,m}\le e\qquad(1\le m\le n-1).
\]
Hence the diagonal entries satisfy
\[
v_p\bigl((C(a))_{m,m}\bigr)=-k_{p,m}\ge -e.
\]

Now let \(1\le r<m\le n-1\). By \eqref{eq:b-expand},
\[
(C(a))_{r,m}
=
\frac{b_{r,m}}{D_m(a)}\frac{C_r(a)}{C_m(a)}.
\]
Since \(b_{r,m}\in\Z\), we have \(v_p(b_{r,m})\ge 0\), and since \(v_p(D_m(a))\le e\),
\[
v_p\!\left(\frac1{D_m(a)}\right)\ge -e.
\]

If \(v_p(a)=0\), then \(p\nmid C_t(a)\) for every \(t\), so
\[
v_p\!\left(\frac{C_r(a)}{C_m(a)}\right)=0.
\]
If \(v_p(a)=t\in\{1,\dots,n-1\}\), then \(p\mid a_t\) and \(p\nmid a_j\) for \(j\neq t\), hence
\[
v_p\!\left(\frac{C_r(a)}{C_m(a)}\right)
=
\Bigl\lfloor\frac{tr}{n}\Bigr\rfloor-\Bigl\lfloor\frac{tm}{n}\Bigr\rfloor.
\]
Because \(1\le r<m\le n-1\) and \(1\le t\le n-1\), both floors lie in \(\{0,\dots,n-2\}\), and
therefore
\[
\Bigl\lfloor\frac{tr}{n}\Bigr\rfloor-\Bigl\lfloor\frac{tm}{n}\Bigr\rfloor
\ge -(n-2).
\]
So in all cases
\[
v_p\bigl((C(a))_{r,m}\bigr)\ge -(e+n-2).
\]

Thus every entry of \(C(a)\) has \(p\)-adic valuation at least \(-(e+n-2)\). Multiplying by
\[
N_n^{\sharp}:=\prod_{p^e\parallel n} p^{\,e+n-2}
\]
clears all denominators, i.e.
\[
N_n^{\sharp}\,C(a)\in M_{n-1}(\Z).
\]
\end{proof}

The periodic datum \(S(a)\) controls the normalized integral-basis pattern, but it need not control the finer unipotent class. The following sextic example makes that distinction completely explicit.

\begin{example}[Global counterexample: \(S(a)\) does not determine the unipotent class]
\label{prop:counterexample-C}
For \(n=6\), let \(a=10\) and \(a'=550\).
Then both \(a\) and \(a'\) are admissible, both satisfy Hypothesis~\textup{(H)}, and
\[
S(a)=S(a'),
\qquad\text{but}\qquad
\Xi(a)\neq \Xi(a').
\]
In particular, the unipotent class \(\Xi(a)\) is not determined by the datum \(S(a)\).
\end{example}

\begin{proof}
Both $10$ and $550$ are sixth-power-free, distinct from $\pm1$, and $x^6-10$, $x^6-550$ are irreducible
by Eisenstein.
Moreover,
\[
550\equiv 10\pmod{36},
\qquad
M(6)=2^2 3^2=36,
\]
so Theorem~\ref{thm:min-periodicity} gives
\[
S(10)=S(550).
\]

Next, both integers satisfy the sextic congruence conditions $A_1$ and $B_2$:
\[
10\equiv 2\pmod4,\quad 10\equiv 1\pmod9,
\qquad
550\equiv 2\pmod4,\quad 550\equiv 1\pmod9.
\]
Hence both are of Type \((A_1,B_2)\) in the notation of
\cite[Table~1]{JakharKalwaniyaRayRoy2026}. For this type,
\cite[Proposition~3.4]{JakharKalwaniyaRayRoy2026} gives the full \(6\times6\)
transition matrix
\[
T_{1,2}(m)=
\begin{pmatrix}
1&0&0&0& C_4(m)/3 &0\\
0&1&0&0&0& C_5(m)/3\\
0&0&1&0& C_2(m)C_4(m)m/3 &0\\
0&0&0&1&0& C_3(m)C_5(m)m/3\\
0&0&0&0&1/3&0\\
0&0&0&0&0&1/3
\end{pmatrix}.
\]
Here \(T_{1,2}(m)\) is the transition matrix between the ordered bases
\[
(1,e_1,\dots,e_5)\qquad\text{and}\qquad (1,\omega_1,\dots,\omega_5).
\]

Applying the map \(\alpha\mapsto \alpha^\perp\) entrywise kills the first basis vector,
since \(1^\perp=0\). Therefore the corresponding transition matrix on trace-zero bases is
obtained by deleting the first row and the first column of \(T_{1,2}(m)\). Thus, with
respect to the ordered bases \((e_1^\perp,\dots,e_5^\perp)\) and
\((b_1^\perp,\dots,b_5^\perp)\), one gets
\[
C(m)=
\begin{pmatrix}
1&0&0&0&C_5(m)/3\\
0&1&0&C_2(m)C_4(m)m/3&0\\
0&0&1&0&C_3(m)C_5(m)m/3\\
0&0&0&1/3&0\\
0&0&0&0&1/3
\end{pmatrix}.
\]

Now compute the strong-decomposition factors.
For $10=2\cdot 5$, one has
\[
C_2(10)=C_3(10)=C_4(10)=C_5(10)=1,
\]
so
\[
C(10)=
\begin{pmatrix}
1&0&0&0&1/3\\
0&1&0&10/3&0\\
0&0&1&0&10/3\\
0&0&0&1/3&0\\
0&0&0&0&1/3
\end{pmatrix}.
\]
For $550=2\cdot 5^2\cdot 11$, one has
\[
C_2(550)=1,\qquad C_3(550)=C_4(550)=C_5(550)=5,
\]
hence
\[
C(550)=
\begin{pmatrix}
1&0&0&0&5/3\\
0&1&0&2750/3&0\\
0&0&1&0&13750/3\\
0&0&0&1/3&0\\
0&0&0&0&1/3
\end{pmatrix}.
\]

Assume for contradiction that
\[
C(550)=C(10)\,U
\qquad\text{for some }U\in U_5^+(\Z).
\]
Since the first three columns of both matrices are the standard basis vectors $E_1,E_2,E_3$, the fourth column of
$C(10)U$ is
\[
\operatorname{col}_4(C(10))+u_{14}E_1+u_{24}E_2+u_{34}E_3.
\]
Its second entry is therefore $10/3+u_{24}$.
But the second entry of the fourth column of $C(550)$ is $2750/3$.
Thus
\[
u_{24}=\frac{2750-10}{3}=\frac{2740}{3}\notin \Z,
\]
a contradiction.
Therefore \(C(10)\) and \(C(550)\) are not right equivalent modulo
\(U_5^+(\Z)\). Thus the corresponding right $U_5^+(\mathbf Z)$-cosets are distinct. 
\end{proof}

It is also worth to note that the datum $S(a)$ is periodic modulo $M(n)$ by Theorem~\ref{thm:min-periodicity}, but the
invariant $\Xi(a)$ need not be.
Example~\ref{prop:counterexample-C} gives explicit sextic parameters $10\equiv 550\pmod{36}=M(6)$ for which
\[
S(10)=S(550)\qquad\text{but}\qquad \Xi(10)\neq \Xi(550).
\]

\begin{remark}[Presentation dependence and the role of \(\Xi(a)\)]
\label{rem:Xi-noncanonical}
\label{rem:holmes-corrected}
\label{rem:sextic-corrected}
The factorization in Corollary~\ref{cor:Xi-factorization} is a convenient
normalization of the chosen pure presentation \(K_a=\Q(\theta)\); it is not a
canonical decomposition of the abstract field.  Both the diagonal data
\((s_m(a))_{m=1}^{n-1}\) and the class \(\Xi(a)\) depend on the presentation
and on the normalized monomial model.  What is intrinsic is the resulting point
of shape space, and later the supporting rational diagonal leaf.

Right multiplication \(C(a)\mapsto C(a)U\) with
\(U\in U_{n-1}^+(\Z)\) corresponds to replacing the ordered \(\Z\)-basis
\((b_1^\perp,\dots,b_{n-1}^\perp)\) of the fixed lattice
\(\cO_{K_a}^\perp\) by another ordered basis obtained by adding integral
combinations of earlier basis vectors.  Hence
\[
(C(a)U)^{\mathsf T}\diag\bigl(s_1(a),\dots,s_{n-1}(a)\bigr)(C(a)U)
=
U^{\mathsf T}
\Bigl(
C(a)^{\mathsf T}\diag\bigl(s_1(a),\dots,s_{n-1}(a)\bigr)C(a)
\Bigr)U,
\]
so only the right \(U_{n-1}^+(\Z)\)-class survives in shape space.
\end{remark}

\begin{remark}[What the periodic datum controls]
\label{rem:Sp-not-enough}
\label{rem:S-versus-C}
Formula~\eqref{eq:b-expand} shows that the coefficients of \(C(a)\) involve the
ratios \(C_r(a)/C_m(a)\), so primes dividing \(a\) re-enter the shape
computation through the strong decomposition.  By contrast, the datum \(S(a)\)
records only the denominator exponents \(k_{p,m}\) and the reduced classes
\(\beta_m \bmod p^{k_{p,m}}\) at primes \(p\mid n\).  Thus \(S(a)\) controls
the normalized integral-basis pattern, but not in general the class
\(\Xi(a)\).  Example~\ref{prop:counterexample-C} shows this explicitly.  In
particular, the periodicity of \(S(a)\) modulo \(M(n)\) does not extend to
\(\Xi(a)\).
\end{remark}

\section{Shape as a complete invariant}\label{sec:shape-complete}

For a number field $K$, write
\[
\operatorname{sig}(K)=(r_1(K),r_2(K))
\]
for its signature.

\subsection{Odd degree}

To extract field-theoretic information from shape, it is convenient to forget the specific normalized basis and work with an arbitrary \(\Z\)-basis of \(\cO_{K_a}^\perp\). The next lemma shows that the same diagonal model still governs every such choice.

\begin{lemma}[Gram factorization in all degrees]
\label{lem:all-gram-factorization}
Let \(n\ge 3\) and let \(a\) be admissible. Let
\[
K_a=\Q(\theta),\qquad \theta^n=a,
\]
and let
\[
\eta=(\eta_1,\dots,\eta_{n-1})
\]
be a \(\Z\)-basis of \(\cO_{K_a}^\perp\).
Let \(G_\eta(a)\) be the Gram matrix of
\[
J(\eta_1),\dots,J(\eta_{n-1})
\]
with respect to \(\langle\cdot,\cdot\rangle_\infty\).
Then there exists a unique matrix \(A_\eta(a)\in \GL_{n-1}(\Q)\) such that
\[
(\eta_1,\dots,\eta_{n-1})=(e_1,\dots,e_{n-1})A_\eta(a).
\]
Consequently,
\begin{equation}\label{eq:all-gram-factorization}
G_\eta(a)=A_\eta(a)^{\mathsf T}\bigl(n\Delta(a)\bigr)A_\eta(a),
\qquad
\Delta(a):=\diag\bigl(s_1(a),\dots,s_{n-1}(a)\bigr).
\end{equation}
\end{lemma}

\begin{proof}
Since \(e_m=\theta^m/C_m(a)\) is a nonzero rational multiple of \(\theta^m\),
the elements \(e_1,\dots,e_{n-1}\) are \(\Q\)-linearly independent.
By Corollary~\ref{cor:diag-gram}, each \(e_m\) has trace zero, so
\[
e_1,\dots,e_{n-1}\in K_a^0.
\]
As \(\dim_\Q K_a^0=n-1\), they form a \(\Q\)-basis of \(K_a^0\).
Hence there is a unique matrix \(A_\eta(a)\in \GL_{n-1}(\Q)\) such that
\[
(\eta_1,\dots,\eta_{n-1})=(e_1,\dots,e_{n-1})A_\eta(a).
\]
By Corollary~\ref{cor:diag-gram}, the Gram matrix of
\[
J(e_1),\dots,J(e_{n-1})
\]
is \(n\Delta(a)\).
Lemma~\ref{lem:gram-congruence} therefore gives
\[
G_\eta(a)=A_\eta(a)^{\mathsf T}\bigl(n\Delta(a)\bigr)A_\eta(a).
\]
\end{proof}

Shape equality forces a rational congruence between two diagonal models, and the first invariant one can read off from that relation is the field generated by diagonal ratios. This is the basic field recovered from the archimedean side alone.

\begin{theorem}[Shape determines the ratio field]\label{thm:shape-ratio-field}
For admissible \(a\), define
\[
F_a:=\Q\!\left(\frac{s_i(a)}{s_j(a)}:1\le i,j\le n-1\right)\subset \R.
\]
If
\[
\operatorname{sh}(K_a)=\operatorname{sh}(K_b),
\]
then
\[
F_a=F_b.
\]
\end{theorem}

\begin{proof}
Choose \(\Z\)-bases
\[
\eta=(\eta_1,\dots,\eta_{n-1})\subset \cO_{K_a}^\perp,\qquad
\nu=(\nu_1,\dots,\nu_{n-1})\subset \cO_{K_b}^\perp,
\]
and let \(G_\eta(a)\), \(G_\nu(b)\) be the corresponding Gram matrices.
Equality of shapes means that there exist \(U\in \GL_{n-1}(\Z)\) and
\(c\in \R_{>0}\) such that
\[
G_\eta(a)=c\,U^{\mathsf T}G_\nu(b)U.
\]

Let \(A_\eta(a),A_\nu(b)\in \GL_{n-1}(\Q)\) be the unique coordinate matrices
given by Lemma~\ref{lem:all-gram-factorization}. Then
\[
G_\eta(a)=A_\eta(a)^{\mathsf T}\bigl(n\Delta(a)\bigr)A_\eta(a),
\qquad
G_\nu(b)=A_\nu(b)^{\mathsf T}\bigl(n\Delta(b)\bigr)A_\nu(b).
\]
Cancelling the common factor \(n\) and conjugating by \(A_\eta(a)^{-1}\)
yields
\[
\Delta(a)=c\,M^{\mathsf T}\Delta(b)M,
\qquad
M:=A_\nu(b)\,U\,A_\eta(a)^{-1}\in \GL_{n-1}(\Q).
\]
Write \(M=(m_{ij})\).
Taking diagonal entries gives, for each \(i\),
\[
s_i(a)=c\sum_{k=1}^{n-1} m_{ki}^2\,s_k(b).
\]

Fix \(j\). Since the \(j\)-th column of \(M\) is nonzero and all
\(s_k(b)>0\), the denominator
\[
\sum_{k=1}^{n-1} m_{kj}^2\,\frac{s_k(b)}{s_1(b)}
\]
is nonzero. Therefore
\[
\frac{s_i(a)}{s_j(a)}
=
\frac{\sum_{k=1}^{n-1} m_{ki}^2\,\dfrac{s_k(b)}{s_1(b)}}
     {\sum_{k=1}^{n-1} m_{kj}^2\,\dfrac{s_k(b)}{s_1(b)}}\in F_b.
\]
Hence \(F_a\subseteq F_b\).
By symmetry, \(F_b\subseteq F_a\), and therefore \(F_a=F_b\).
\end{proof}

In the pure setting, that abstract ratio field has a very concrete description. The next corollary identifies it with the obvious radical field generated by the basic archimedean scale.

\begin{corollary}[Explicit description of the ratio field]\label{cor:ratio-field-explicit}
For every admissible \(a\),
\[
F_a=\Q\bigl(|a|^{2/n}\bigr).
\]
\end{corollary}

\begin{proof}
Since
\[
s_m(a)=\frac{|a|^{2m/n}}{C_m(a)^2},
\]
every ratio \(s_i(a)/s_j(a)\) belongs to \(\Q(|a|^{2/n})\), so
\[
F_a\subseteq \Q(|a|^{2/n}).
\]
On the other hand, \(C_1(a)=1\), and therefore
\[
\frac{s_2(a)}{s_1(a)}
=
\frac{|a|^{4/n}/C_2(a)^2}{|a|^{2/n}}
=
\frac{|a|^{2/n}}{C_2(a)^2}.
\]
Hence
\[
|a|^{2/n}=C_2(a)^2\,\frac{s_2(a)}{s_1(a)}\in F_a.
\]
Thus \(F_a=\Q(|a|^{2/n})\).
\end{proof}

Once the ratio field is known explicitly, odd degree becomes rigid because squaring does not lose information when \(\gcd(2,n)=1\). This is where the parity dichotomy first appears sharply.

\begin{corollary}[Odd degree]\label{cor:odd-complete-short}
Assume that \(n\) is odd.
If
\[
\operatorname{sh}(K_a)=\operatorname{sh}(K_b),
\]
then
\[
K_a\simeq K_b.
\]
\end{corollary}

\begin{proof}
By Theorem~\ref{thm:shape-ratio-field} and Corollary~\ref{cor:ratio-field-explicit},
\[
\Q\bigl(|a|^{2/n}\bigr)=\Q\bigl(|b|^{2/n}\bigr).
\]
Let \(\beta_a\) be the unique real root of \(x^n-a\). Then \(\beta_a^2=|a|^{2/n}\).
Since \(\gcd(2,n)=1\), there exist integers \(u,v\) with \(2u+nv=1\), and therefore
\[
\beta_a=(\beta_a^2)^u a^v\in \Q(\beta_a^2)=\Q\bigl(|a|^{2/n}\bigr).
\]
Thus
\[
K_a=\Q(\beta_a)=\Q\bigl(|a|^{2/n}\bigr).
\]
Similarly,
\[
K_b=\Q\bigl(|b|^{2/n}\bigr).
\]
Hence \(K_a\simeq K_b\).
\end{proof}

In even degree the same argument no longer recovers the full field, but it still recovers a canonical core radicand field. This is the unconditional part of the even-degree story.

\begin{corollary}[Even degree: the unconditional consequence]\label{cor:even-core}
Assume that \(n=2r\) is even.
If
\[
\operatorname{sh}(K_a)=\operatorname{sh}(K_b),
\]
then
\[
\Q\bigl(|a|^{1/r}\bigr)=\Q\bigl(|b|^{1/r}\bigr).
\]
Equivalently,
\[
\Q\bigl(|a|^{2/n}\bigr)=\Q\bigl(|b|^{2/n}\bigr).
\]
\end{corollary}

\begin{proof}
This is immediate from Corollary~\ref{cor:ratio-field-explicit}.
\end{proof}

\begin{remark}[What the shape recovers in even degree]\label{rem:even-obstruction}
Let $n=2r$.
The field recovered from the diagonal ratios is
\[
F_a=\Q\bigl(|a|^{2/n}\bigr)=\Q\bigl(|a|^{1/r}\bigr).
\]
This is the unconditional conclusion.
One should not, however, state that $F_a$ is always a subfield (or even a proper subfield) of $K_a$.

Indeed, if $a>0$, then $\theta^2\in K_a$ and
\[
(\theta^2)^r=a=|a|,
\]
so $F_a=\Q(\theta^2)\subset K_a$.
If $a<0$ and $r$ is odd, then $-\theta^2\in K_a$ and
\[
(-\theta^2)^r=(-1)^r\theta^{2r}=(-1)^r a=|a|,
\]
so again $F_a=\Q(-\theta^2)\subset K_a$.
By contrast, if $a<0$ and $r$ is even, then
\[
(\theta^2)^r=a=-|a|,
\]
and obtaining an $r$th root of $|a|$ from $\theta^2$ would require multiplying by an $r$th root of $-1$.
There is no general reason for such an element to lie in $K_a$.
Thus the correct general statement is that the shape determines the radicand field
$\Q(|a|^{1/r})$; the preceding argument does not identify it, in general, as an actual subfield of $K_a$.
\end{remark}

\subsection{Even degree}
\label{subsec:even-sign-squarefree}

Fix an even integer
\[
n=2r\ge 4.
\]

To show that the even-degree obstruction is genuine rather than a formal defect of the proof, we need an explicit supply of pure fields whose rings of integers are given by the power basis. The following criterion is exactly the tool for that.

\begin{theorem}[Nguyen-Dang--Nguyen \(\alpha\)-monogeneity criterion]
\label{thm:NDH-monogenic}
Let \(K=\Q(\alpha)\) with \(\alpha\) a root of an irreducible polynomial
\[
x^n-m\in \Z[x].
\]
Then
\[
\cO_K=\Z[\alpha]
\]
if and only if \(m\) is squarefree and
\[
v_p(m^p-m)=1
\qquad\text{for every prime }p\mid n.
\]
\end{theorem}

\begin{proof}
This is \cite[Theorem~2.7]{NguyenDangHung2025Alpha}.
\end{proof}

We also use the following even-degree core-field consequence established earlier in the paper:
for admissible parameters \(a,b\),
\begin{equation}\label{eq:even-core-field-consequence}
\sh(K_a)=\sh(K_b)
\quad\Longrightarrow\quad
\Q\!\bigl(|a|^{1/r}\bigr)=\Q\!\bigl(|b|^{1/r}\bigr).
\end{equation}

The next theorem shows that even-degree shape really can forget sign in infinitely many cases. In other words, the ambiguity detected above is not an artifact of the method but a genuine geometric phenomenon.

\begin{theorem}[Sign ambiguity for the shape]\label{thm:even-Sh-ambiguity}
Assume that $n=2r \geq 4$ is even.
Then there exist infinitely many squarefree integers $a>0$ such that
\[
\operatorname{sh}(K_a)=\operatorname{sh}(K_{-a})
\qquad\text{but}\qquad
K_a\not\simeq K_{-a}.
\]
\end{theorem}

\begin{proof}
Set
\[
N:=4\prod_{\substack{p\mid n\\ p\ \mathrm{odd}}} p^2.
\]
For each odd prime \(p\mid n\), choose a residue class \(c_p\pmod{p^2}\) satisfying
\[
2c_p\equiv 1+p \pmod{p^2}.
\]
Also impose the congruence
\[
c\equiv 1 \pmod 4.
\]
By the Chinese remainder theorem, there exists an integer \(c\) modulo \(N\) satisfying
\[
c\equiv c_p \pmod{p^2}\quad (p\mid n,\ p\ \mathrm{odd}),
\qquad
c\equiv 1\pmod 4.
\]
Since \(c\) is coprime to \(N\), Dirichlet's theorem yields infinitely many primes
\[
q\equiv c \pmod N.
\]
Fix such a prime \(q\) and put
\[
a:=2q.
\]

Then \(a>0\) and \(a\) is squarefree. Hence \(a\) is \(n\)th-power-free, and for every prime
\(p\mid n\) we have \(v_p(a)\in\{0,1\}\); therefore Hypothesis~\textup{(H)} holds.
Moreover, since \(q\mid a\) but \(q^2\nmid a\), both
\[
x^n-a\qquad\text{and}\qquad x^n+a
\]
are Eisenstein at \(q\), hence irreducible over \(\Q\). Thus \(a\) and \(-a\) are admissible.

We now verify the criterion.

If \(p=2\), then \(a=2q\) with \(q\) odd, so
\[
v_2(a^2-a)=v_2\bigl(2q(2q-1)\bigr)=1,
\]
and
\[
v_2((-a)^2-(-a))=v_2(a^2+a)=v_2\bigl(2q(2q+1)\bigr)=1.
\]

Now let \(p\mid n\) be odd. Since \(q\equiv c\pmod{p^2}\), we have
\[
a=2q\equiv 2c\equiv 1+p \pmod{p^2}.
\]
In particular \(p\nmid a\), and
\[
a^{p-1}\equiv (1+p)^{p-1}\equiv 1+(p-1)p \equiv 1-p \not\equiv 1 \pmod{p^2}.
\]
Since \(p\nmid a\), Fermat's little theorem gives
\[
a^{p-1}\equiv 1 \pmod p,
\]
while the displayed congruence shows
\[
a^{p-1}\not\equiv 1 \pmod{p^2}.
\]
Hence
\[
v_p(a^{p-1}-1)=1,
\]
and therefore
\[
v_p(a^p-a)=v_p(a)+v_p(a^{p-1}-1)=1.
\]
Since \(p\) is odd,
\[
(-a)^p-(-a)=-(a^p-a),
\]
hence also
\[
v_p((-a)^p-(-a))=1.
\]

Thus Theorem~\ref{thm:NDH-monogenic} gives
\[
\cO_{K_a}=\Z[\theta],\qquad \theta^n=a,
\]
and
\[
\cO_{K_{-a}}=\Z[\vartheta],\qquad \vartheta^n=-a.
\]

Because \(a\) is squarefree, the strong decomposition of \(a\) has
\[
a_1=a,\qquad a_j=1\quad (2\le j\le n-1),
\]
and the strong decomposition of \(-a\) has the same positive part. Hence
\[
C_m(a)=C_m(-a)=1
\qquad (1\le m\le n-1).
\]
Therefore the normalized monomials in the two fields are simply
\[
e_m(a)=\theta^m,
\qquad
e_m(-a)=\vartheta^m
\qquad (1\le m\le n-1).
\]

Since \(\Tr(\theta^m)=\Tr(\vartheta^m)=0\) for \(1\le m\le n-1\), Lemma~\ref{lem:exact-sequence}
shows that
\[
(n\theta,\dots,n\theta^{n-1})
\]
is a \(\Z\)-basis of \(\cO_{K_a}^{\perp}\), and
\[
(n\vartheta,\dots,n\vartheta^{n-1})
\]
is a \(\Z\)-basis of \(\cO_{K_{-a}}^{\perp}\).

By Proposition~\ref{prop:orthogonality}, for \(1\le i,j\le n-1\),
\[
\bigl\langle J(\theta^i),J(\theta^j)\bigr\rangle_\infty
=
\begin{cases}
0,& i\neq j,\\[2mm]
n\,a^{2i/n},& i=j,
\end{cases}
\]
and exactly the same formula holds with \(\vartheta\) in place of \(\theta\), because only \(|a|\)
enters. Hence the Gram matrices of the above trace-zero bases are identical:
\[
\mathrm{Gram}\bigl(J(n\theta),\dots,J(n\theta^{n-1})\bigr)
=
\mathrm{Gram}\bigl(J(n\vartheta),\dots,J(n\vartheta^{n-1})\bigr)
=
n^3\diag\!\bigl(a^{2/n},a^{4/n},\dots,a^{2(n-1)/n}\bigr).
\]
Therefore
\[
\sh(K_a)=\sh(K_{-a}).
\]

Finally, since \(n=2r\) is even and \(a>0\), the polynomial \(x^n-a\) has exactly two real roots,
whereas \(x^n+a\) has no real roots. Thus
\[
\operatorname{sig}(K_a)=(2,r-1),
\qquad
\operatorname{sig}(K_{-a})=(0,r),
\]
so \(K_a\not\simeq K_{-a}\).
Since there are infinitely many such primes \(q\), the theorem follows.
\end{proof}

To recover rigidity on the squarefree subfamily, we use a standard torsion
theorem for radical extensions.  In the real pure-field situation needed below,
it gives the following precise replacement for the radical-elements statement.

\begin{lemma}[Radical torsion in a real pure field]
\label{lem:pure-generators}
Let \(r\ge 2\), let \(c\in \Q_{>0}\), and let \(\lambda\) be the positive real
root of \(X^r-c\).  Assume
\[
[\Q(\lambda):\Q]=r.
\]
Put
\[
F:=\Q(\lambda)\subset \R.
\]
If \(x\in F^\times\) satisfies
\[
x^r\in \Q^\times,
\]
then
\[
x=q\,\lambda^u
\]
for some \(q\in \Q^\times\) and some integer \(u\) with \(0\le u\le r-1\).
If moreover \(\Q(x)=F\), then
\[
\gcd(u,r)=1.
\]
\end{lemma}

\begin{proof}
Write \(T(F^\times/\Q^\times)\) for the torsion subgroup of
\(F^\times/\Q^\times\).  The coset \(\lambda\Q^\times\) has exact order \(r\):
indeed, if \(\lambda^d\in\Q^\times\) for some \(1\le d<r\), then \(\lambda\)
satisfies a polynomial of degree \(d\) over \(\Q\), contradicting
\([\Q(\lambda):\Q]=r\).

Let
\[
M:=\langle \Q^\times,\lambda\rangle\subset F^\times.
\]
Then
\[
M/\Q^\times=\langle \lambda\Q^\times\rangle
\]
is cyclic of order \(r\), and \(\Q(M)=F\).  Since \(F\subset \R\), no primitive
\(2p\)-th root of unity lies in \(F\setminus \Q\), for any prime \(p\).  Hence
\cite[Theorem~1.7]{GayVelez1981}, applied with base field \(\Q\), gives
\[
T(F^\times/\Q^\times)=M/\Q^\times.
\]

Now \(x^r\in\Q^\times\), so the coset \(x\Q^\times\) is torsion in
\(F^\times/\Q^\times\).  Therefore
\[
x\Q^\times=(\lambda\Q^\times)^u
\]
for some \(u\in\Z\).  Reducing \(u\) modulo \(r\), we may assume
\(0\le u\le r-1\).  Thus
\[
x=q\,\lambda^u
\]
for some \(q\in\Q^\times\).

Finally assume \(\Q(x)=F\), and let \(d:=\gcd(u,r)\).  If \(d>1\), then
\[
x\in \Q(\lambda^d).
\]
But \(\lambda^d\) satisfies
\[
X^{r/d}-c\in \Q[X],
\]
so
\[
[\Q(\lambda^d):\Q]\le r/d<r=[F:\Q].
\]
Therefore \(\Q(x)\subseteq \Q(\lambda^d)\neq F\), a contradiction.  Hence
\(d=1\).
\end{proof}

Among squarefree radicands, the core field remembers more than just a general radical extension: it remembers the radicand itself. The next lemma isolates that uniqueness statement.

\begin{lemma}[Squarefree core uniqueness]
\label{lem:squarefree-core-uniqueness}
Let \(r\ge 2\), and let \(u,v>1\) be squarefree integers such that
\[
X^r-u
\qquad\text{and}\qquad
X^r-v
\]
are irreducible over \(\Q\).  If
\[
\Q(u^{1/r})=\Q(v^{1/r}),
\]
where \(u^{1/r}\) and \(v^{1/r}\) denote the positive real roots, then
\[
u=v.
\]
\end{lemma}

\begin{proof}
Let
\[
\lambda:=u^{1/r},
\qquad
\mu:=v^{1/r},
\qquad
F:=\Q(\lambda)=\Q(\mu).
\]
Since \(X^r-v\) is irreducible, \([\Q(\mu):\Q]=r\), and hence
\(\Q(\mu)=F\).  Applying Lemma~\ref{lem:pure-generators} to the real pure field
\(F=\Q(\lambda)\), we obtain
\[
\mu=q\,\lambda^k
\]
for some \(q\in\Q^\times\) and some integer \(k\) satisfying
\[
1\le k\le r-1,
\qquad
\gcd(k,r)=1.
\]
Since \(\lambda>0\) and \(\mu>0\), one has \(q>0\).  Raising to the \(r\)-th
power gives
\[
v=q^r u^k.
\]

Let \(\ell\) be a prime divisor of \(u\).  Since \(u\) is squarefree,
\[
v_\ell(v)=k+r\,v_\ell(q).
\]
Because \(v\) is squarefree, \(v_\ell(v)\in\{0,1\}\).  Since
\(1\le k\le r-1\), the only possibility is
\[
k=1,
\qquad
v_\ell(q)=0.
\]
Thus every prime divisor of \(u\) divides \(v\) with valuation \(1\), and
\[
\operatorname{Supp}(u)\subseteq \operatorname{Supp}(v).
\]

Now let \(\ell\) be a prime divisor of the numerator or denominator of \(q\).
If \(\ell\nmid u\), then
\[
v_\ell(v)=r\,v_\ell(q),
\]
which is either negative or has absolute value at least \(r\), impossible
because \(v\) is a positive squarefree integer.  If \(\ell\mid u\), then we
already proved \(v_\ell(q)=0\).  Hence \(v_\ell(q)=0\) for every prime
\(\ell\), so \(q=1\).  Therefore
\[
v=u.
\]
\end{proof}

This yields the clean even-degree classification promised in the introduction: on squarefree admissible parameters, shape determines the field up to the unavoidable sign ambiguity.

\begin{theorem}
\label{thm:shape-complete-squarefree-up-to-sign}
Assume \(n=2r\ge 4\) is even. Let \(a,b\) be squarefree admissible integers.
If
\[
\sh(K_a)=\sh(K_b),
\]
then
\[
|a|=|b|.
\]
Equivalently,
\[
K_b\simeq K_a
\qquad\text{or}\qquad
K_b\simeq K_{-a}.
\]
\end{theorem}

\begin{proof}
By the even-degree core-field consequence \eqref{eq:even-core-field-consequence},
\[
\sh(K_a)=\sh(K_b)
\quad\Longrightarrow\quad
\Q\!\bigl(|a|^{1/r}\bigr)=\Q\!\bigl(|b|^{1/r}\bigr).
\]
Since \(|a|\) and \(|b|\) are squarefree integers \(>1\), the polynomials
\[
x^r-|a|
\qquad\text{and}\qquad
x^r-|b|
\]
are irreducible over \(\Q\) by Eisenstein's criterion.
Therefore Lemma~\ref{lem:squarefree-core-uniqueness} applies and gives
\[
|a|=|b|.
\]
Hence \(b=\pm a\).
If \(b=a\), then \(K_b=K_a\).
If \(b=-a\), then \(K_b=K_{-a}\).
\end{proof}

\begin{remark}
Combining Theorems~\ref{thm:even-Sh-ambiguity}
and~\ref{thm:shape-complete-squarefree-up-to-sign}, we obtain infinitely many squarefree admissible
parameters \(a>0\) for which the fiber of the shape map inside the squarefree admissible family is exactly
\[
\{\,K_a,\ K_{-a}\,\}.
\]
Thus, on the squarefree admissible subfamily in even degree, the Minkowski shape has exactly the residual
ambiguity predicted by sign.
\end{remark}

Once the only remaining ambiguity is sign, the signature separates the two possibilities immediately. Thus adding signature restores full completeness in even degree.

\begin{corollary}
\label{cor:sig-sh-complete-squarefree}
Assume \(n=2r\ge 4\) is even. Let \(a,b\) be squarefree admissible integers.
If
\[
\operatorname{sig}(K_a)=\operatorname{sig}(K_b)
\qquad\text{and}\qquad
\operatorname{sh}(K_a)=\operatorname{sh}(K_b),
\]
then
\[
K_a\simeq K_b.
\]
\end{corollary}

\begin{proof}
By Theorem~\ref{thm:shape-complete-squarefree-up-to-sign},
\[
\operatorname{sh}(K_a)=\operatorname{sh}(K_b)
\quad\Longrightarrow\quad
|a|=|b|.
\]
Put \(c:=|a|=|b|\). Then \(\{K_a,K_b\}=\{K_c,K_{-c}\}\).
Since \(n=2r\) is even and \(c>0\), the polynomial \(x^n-c\) has exactly two
real roots, whereas \(x^n+c\) has no real roots. Hence
\[
\operatorname{sig}(K_c)=(2,r-1),
\qquad
\operatorname{sig}(K_{-c})=(0,r).
\]
Therefore equal signatures force \(K_a\simeq K_b\).
\end{proof}

\section{Discriminant and divisor products}\label{sec:disc-divisor}

In this section, let $a$ be admissible and write its strong decomposition as
\begin{equation}\label{eq:strong-decomp-disc}
a=\varepsilon\prod_{j=1}^{n-1} a_j^{\,j},
\qquad
\varepsilon\in\{\pm1\},
\qquad
a_j\in \Z_{>0}\text{ squarefree and pairwise coprime}.
\end{equation}
We continue to write
\[
C_m(a)=\prod_{j=1}^{n-1} a_j^{\lfloor jm/n\rfloor}
\qquad (0\le m\le n-1)
\]
as in \eqref{eq:Cm-def}. The point of the section is that the same data governing shape also control the discriminant, but in a different combination. What emerges is a divisor-lattice expansion in which the archimedean parameters and the finite periodic factors play visibly different roles.

\subsection{Discriminant identities}

To turn the strong decomposition into an explicit discriminant formula, we need to understand the total contribution of the floor terms appearing in the normalization factors \(C_m(a)\). The next lemma is the exact combinatorial identity that drives the calculation.

\begin{lemma}[A floor-sum identity]\label{lem:floor-sum-exact}
Let $1\le j\le n-1$ and put $g=\gcd(j,n)$.
Then
\[
\sum_{m=1}^{n-1}\Bigl\lfloor\frac{jm}{n}\Bigr\rfloor
=\frac{(j-1)(n-1)+(g-1)}{2}.
\]
\end{lemma}

\begin{proof}
Write
\[
j=gj_0,
\qquad
n=gn_0,
\qquad
\gcd(j_0,n_0)=1.
\]
For $1\le m\le n-1$, write either $m=qn_0$ with $1\le q\le g-1$, or
\[
m=qn_0+r
\qquad
(0\le q\le g-1,\,1\le r\le n_0-1).
\]
Hence
\begin{align*}
\sum_{m=1}^{n-1}\Bigl\lfloor\frac{jm}{n}\Bigr\rfloor
&=\sum_{q=0}^{g-1}\sum_{r=1}^{n_0-1}\Bigl\lfloor\frac{j(qn_0+r)}{n}\Bigr\rfloor
   +\sum_{q=1}^{g-1}\Bigl\lfloor\frac{jqn_0}{n}\Bigr\rfloor \\
&=\sum_{q=0}^{g-1}\sum_{r=1}^{n_0-1}\left(qj_0+\Bigl\lfloor\frac{j_0r}{n_0}\Bigr\rfloor\right)
   +\sum_{q=1}^{g-1} qj_0 \\
&=g\sum_{r=1}^{n_0-1}\Bigl\lfloor\frac{j_0r}{n_0}\Bigr\rfloor
  +n_0j_0\sum_{q=1}^{g-1} q.
\end{align*}
Because $\gcd(j_0,n_0)=1$, the residues $j_0,2j_0,\dots,(n_0-1)j_0$ modulo $n_0$ are a permutation of
$1,2,\dots,n_0-1$.
Therefore
\begin{align*}
\sum_{r=1}^{n_0-1}\Bigl\lfloor\frac{j_0r}{n_0}\Bigr\rfloor
&=\sum_{r=1}^{n_0-1}\frac{j_0r-(j_0r\bmod n_0)}{n_0} \\
&=\frac{j_0\sum_{r=1}^{n_0-1} r-\sum_{s=1}^{n_0-1} s}{n_0}
=\frac{(j_0-1)(n_0-1)}{2}.
\end{align*}
Substituting this and
\[
\sum_{q=1}^{g-1}q=\frac{g(g-1)}{2}
\]
into the previous formula gives
\begin{align*}
\sum_{m=1}^{n-1}\Bigl\lfloor\frac{jm}{n}\Bigr\rfloor
&=\frac{g(j_0-1)(n_0-1)}{2}+\frac{n_0j_0g(g-1)}{2} \\
&=\frac{jn-j-n+g}{2}
=\frac{(j-1)(n-1)+(g-1)}{2}.
\end{align*}
\end{proof}

We also record the standard change-of-basis rule for discriminants, since it will let us move cleanly from the power basis to the normalized integral basis.

\begin{lemma}[Discriminant under basis change]\label{lem:disc-basis-change}
Let $K/\Q$ be a number field of degree $n$.
If $v=(v_0,\dots,v_{n-1})$ is a $\Q$-basis of $K$ and $v'=vT$ with $T\in \GL_n(\Q)$, then
\[
\disc(v')=\disc(v)\,(\det T)^2.
\]
\end{lemma}

\begin{proof}
Let $G(v)=(\Tr(v_iv_j))_{0\le i,j\le n-1}$ be the Gram matrix of the trace pairing in the basis $v$.
Then
\[
G(v')=T^{\mathsf T}G(v)T,
\]
so
\[
\disc(v')=\det G(v')=(\det T)^2\det G(v)=(\det T)^2\disc(v).
\]
\end{proof}

Assume now that Hypothesis~\textup{(H)} holds, so that the normalized integral basis
\begin{equation}\label{eq:ND-basis-disc}
\mathcal B_a=
\Bigl(1,\ \omega_m:=\frac{\theta^m+\beta_m}{C_m(a)D_m(a)}\ (1\le m\le n-1)\Bigr)
\end{equation}
of Theorem~\ref{thm:ND-basis} exists. Applying the previous lemma to the integral basis isolates the exact contribution of the normalization factors \(C_m(a)\) and \(D_m(a)\). This is the algebraic starting point for all later discriminant formulas.

\begin{proposition}[Discriminant decomposition in the integral basis]
\label{prop:disc-ND-basis}
Assume Hypothesis~\textup{(H)}.
Then
\[
\disc(K_a)=\disc(1,\theta,\dots,\theta^{n-1})
\cdot
\left(\prod_{m=1}^{n-1}\frac{1}{C_m(a)D_m(a)}\right)^2.
\]
\end{proposition}

\begin{proof}
Express the ordered basis $\mathcal B_a$ of \eqref{eq:ND-basis-disc} in the power basis
$(1,\theta,\dots,\theta^{n-1})$.
Since each $\beta_m$ is a $\Z$-linear combination of $1,\theta,\dots,\theta^{m-1}$, the change-of-basis matrix
is upper triangular with diagonal entries
\[
1,\ (C_1(a)D_1(a))^{-1},\dots,(C_{n-1}(a)D_{n-1}(a))^{-1}.
\]
Hence its determinant equals
\[
\prod_{m=1}^{n-1}(C_m(a)D_m(a))^{-1}.
\]
Applying Lemma~\ref{lem:disc-basis-change} and using that $\mathcal B_a$ is an integral basis of $\cO_{K_a}$ gives
the result.
\end{proof}

For the power basis one has
\[
\disc(1,\theta,\dots,\theta^{n-1})
=
\disc(X^n-a)
=
(-1)^{\frac{(n-1)(n-2)}{2}}\,n^n a^{\,n-1}.
\]
Indeed, if \(f(X)=X^n-a\), then for a monic polynomial
\[
\disc(f)=(-1)^{\frac{n(n-1)}2}\operatorname{Res}(f,f').
\]
Since \(f'(X)=nX^{n-1}\), and the product of the roots of \(X^n-a\) is
\((-1)^{n+1}a\), we get
\[
\operatorname{Res}(f,f')
=
\prod_{f(\alpha)=0} f'(\alpha)
=
n^n\left(\prod_{f(\alpha)=0}\alpha\right)^{n-1}
=
n^n\bigl((-1)^{n+1}a\bigr)^{n-1}.
\]
Thus
\[
\disc(X^n-a)
=
(-1)^{\frac{n(n-1)}2+(n+1)(n-1)}n^n a^{n-1}
=
(-1)^{\frac{(n-1)(n-2)}2}n^n a^{n-1}.
\]
Taking absolute values therefore yields
\begin{equation}\label{eq:disc-raw-revised}
|\disc(K_a)|
=
n^n |a|^{n-1}\cdot
\prod_{m=1}^{n-1} C_m(a)^{-2}\cdot
\prod_{m=1}^{n-1} D_m(a)^{-2}.
\end{equation}

The basis-change formula becomes much more transparent when rewritten in strong-decomposition coordinates: each squarefree block \(a_j\) contributes with an exponent depending only on \(\gcd(j,n)\). This is the first appearance of the divisor-lattice pattern.

\begin{corollary}[Uniform exponent vector in the strong decomposition]
\label{cor:disc-exponent-vector-revised}
Assume Hypothesis~\textup{(H)} and define
\[
w_j:=n-\gcd(j,n)
\qquad (1\le j\le n-1).
\]
Then
\begin{equation}\label{eq:disc-factor-revised}
|\disc(K_a)|=\kappa_n\bigl(S(a)\bigr)\cdot \prod_{j=1}^{n-1} a_j^{\,w_j},
\end{equation}
where
\[
\kappa_n\bigl(S(a)\bigr):=n^n\prod_{m=1}^{n-1} D_m(a)^{-2}.
\]
In particular, $\kappa_n\bigl(S(a)\bigr)$ depends only on the datum $S(a)$, and hence only on the
residue class of $a\bmod M(n)$.
\end{corollary}

\begin{proof}
From \eqref{eq:Cm-def},
\[
\prod_{m=1}^{n-1} C_m(a)^2
=\prod_{j=1}^{n-1} a_j^{\,2\sum_{m=1}^{n-1}\lfloor jm/n\rfloor}.
\]
By Lemma~\ref{lem:floor-sum-exact}, the exponent of $a_j$ in the latter product equals
\[
2\sum_{m=1}^{n-1}\Bigl\lfloor\frac{jm}{n}\Bigr\rfloor
=(j-1)(n-1)+(\gcd(j,n)-1).
\]
The exponent of $a_j$ in $|a|^{n-1}=\prod_j a_j^{j(n-1)}$ is $j(n-1)$.
Subtracting exponents in \eqref{eq:disc-raw-revised} gives
\[
j(n-1)-\bigl((j-1)(n-1)+(\gcd(j,n)-1)\bigr)=n-\gcd(j,n)=w_j.
\]
This proves \eqref{eq:disc-factor-revised}.
\end{proof}

There is also a complementary archimedean reading of the same formula: the discriminant can be reconstructed from the diagonal parameters \(s_m(a)\) together with the discrete determinant factor \(\det C(a)\). This is the discriminant analogue of the shape factorization.

\begin{proposition}[Discriminant from the diagonal parameters and the discrete factor]
\label{prop:disc-from-s-revised}
Assume that $a$ is admissible and satisfies Hypothesis~\textup{(H)}.
Then
\begin{equation}\label{eq:disc-from-s-revised}
|\disc(K_a)|
=n^n\left(\prod_{m=1}^{n-1}s_m(a)\right)\bigl(\det C(a)\bigr)^2.
\end{equation}
In particular, because
\[
\det C(a)=\prod_{m=1}^{n-1} D_m(a)^{-1}
\]
by Lemma~\ref{lem:C-matrix}\textup{(b)}, the factor $\bigl(\det C(a)\bigr)^2$ depends only on the datum $S(a)$.
\end{proposition}

\begin{proof}
By definition,
\[
\prod_{m=1}^{n-1}s_m(a)
=\prod_{m=1}^{n-1}\frac{|a|^{2m/n}}{C_m(a)^2}
=\frac{|a|^{\frac{2}{n}\sum_{m=1}^{n-1}m}}{\prod_{m=1}^{n-1}C_m(a)^2}
=\frac{|a|^{n-1}}{\prod_{m=1}^{n-1}C_m(a)^2}.
\]
Substituting this into \eqref{eq:disc-raw-revised} gives
\[
|\disc(K_a)|
=n^n\left(\prod_{m=1}^{n-1}s_m(a)\right)\left(\prod_{m=1}^{n-1}D_m(a)^{-2}\right).
\]
By Lemma~\ref{lem:C-matrix}\textup{(b)},
\[
\prod_{m=1}^{n-1}D_m(a)^{-2}=\bigl(\det C(a)\bigr)^2,
\]
which proves \eqref{eq:disc-from-s-revised}.
\end{proof}

\begin{remark}[Consistency with the covolume formula]\label{rem:disc-covolume-consistency}
Combining Proposition~\ref{prop:disc-from-s-revised} with Theorem~\ref{thm:factorization} and
Lemma~\ref{lem:G0perp} recovers Proposition~\ref{prop:covolume-tracezero}:
\[
\mathrm{covol}\bigl(J(\cO_{K_a}^\perp)\bigr)^2=n^{2n-3}|\disc(K_a)|.
\]
Thus the discriminant decomposition above is fully consistent with the unconditional covolume computation already
proved in Section~\ref{sec:minkowski-shape}.
\end{remark}

\subsection{Symmetrized diagonal invariants and divisor products}

To connect the discriminant formulas with the later geometric coordinates, it is helpful to rewrite the diagonal parameters \(s_m(a)\) directly in terms of the strong decomposition.

\begin{lemma}[Strong-decomposition formula for $s_m(a)$]\label{lem:s-strong}
For $1\le m\le n-1$ one has
\begin{equation}\label{eq:s-strong}
s_m(a)=\prod_{j=1}^{n-1} a_j^{\,2\{jm/n\}},
\end{equation}
where $\{x\}=x-\lfloor x\rfloor$ denotes the fractional part.
\end{lemma}

\begin{proof}
By \eqref{eq:strong-decomp-disc},
\[
|a|^{2m/n}=\prod_{j=1}^{n-1} a_j^{2jm/n}.
\]
Dividing by
\[
C_m(a)^2=\prod_{j=1}^{n-1} a_j^{2\lfloor jm/n\rfloor}
\]
gives
\[
s_m(a)=\prod_{j=1}^{n-1} a_j^{2jm/n-2\lfloor jm/n\rfloor}
=\prod_{j=1}^{n-1} a_j^{2\{jm/n\}}.
\]
\end{proof}

The products \(s_m(a)s_{n-m}(a)\) simplify dramatically: the asymmetry in the fractional-part formula disappears, and only divisibility conditions remain. This is where the divisor-lattice structure becomes completely visible.

\begin{lemma}[Symmetrized diagonal products depend only on $\gcd(m,n)$]
\label{lem:symmetrized-products-revised}
For $1\le m\le n-1$, define
\[
\Sigma_m(a):=s_m(a)s_{n-m}(a).
\]
Then
\begin{equation}\label{eq:symmetrized-products-revised}
\Sigma_m(a)=\prod_{\substack{1\le j\le n-1\\ n\nmid mj}} a_j^{\,2}.
\end{equation}
Equivalently, if $g=\gcd(m,n)$ and $d=n/g$, then
\begin{equation}\label{eq:symmetrized-products-divisor}
\Sigma_m(a)=\prod_{\substack{1\le j\le n-1\\ d\nmid j}} a_j^{\,2}.
\end{equation}
In particular, $\Sigma_m(a)$ depends only on $\gcd(m,n)$.
\end{lemma}

\begin{proof}
By Lemma~\ref{lem:s-strong},
\[
\Sigma_m(a)
=\prod_{j=1}^{n-1} a_j^{\,2\{jm/n\}+2\{j(n-m)/n\}}.
\]
Fix $j$ and write $x=jm/n$.
Since
\[
\frac{j(n-m)}{n}=j-x,
\]
and $j\in \Z$, one has
\[
\Bigl\{\frac{j(n-m)}{n}\Bigr\}=\{j-x\}=\{-x\}.
\]
If $x\notin \Z$ (equivalently, $n\nmid mj$), then $\{-x\}=1-\{x\}$, so the exponent of $a_j$ is
\[
2\{x\}+2(1-\{x\})=2.
\]
If $x\in \Z$ (equivalently, $n\mid mj$), then both fractional parts vanish, so the exponent is $0$.
This proves \eqref{eq:symmetrized-products-revised}.

Now write $m=gm'$ and $n=gn'$ with $\gcd(m',n')=1$.
Then
\[
n\mid mj
\iff gn'\mid gm'j
\iff n'\mid j,
\]
so the condition $n\nmid mj$ is equivalent to $n'=d\nmid j$.
This proves \eqref{eq:symmetrized-products-divisor}.
\end{proof}

As a consequence, the arithmetic divisor-products \(B_d(a)\) can already be read off from simple combinations of the archimedean diagonal invariants.

\begin{corollary}[Divisor-products recovered from the symmetrized diagonal invariants]
\label{cor:divisor-products-revised}
For every divisor $d\mid n$, define
\[
B_d(a):=\prod_{\substack{1\le j\le n-1\\ d\mid j}} a_j\in \Z_{>0}.
\]
Then:
\begin{enumerate}[label=\textup{(\roman*)},leftmargin=2.4em]
\item
\begin{equation}\label{eq:B1-sigma1}
B_1(a)^2=\Sigma_1(a)=s_1(a)s_{n-1}(a);
\end{equation}
\item for every divisor $d$ with $1<d\le n$,
\begin{equation}\label{eq:Bd-sigma}
B_d(a)^2=\frac{\Sigma_1(a)}{\Sigma_{n/d}(a)}
=\frac{s_1(a)s_{n-1}(a)}{s_{n/d}(a)s_{n-n/d}(a)}.
\end{equation}
\end{enumerate}
\end{corollary}

\begin{proof}
By Lemma~\ref{lem:symmetrized-products-revised} with $m=1$, we have
\[
\Sigma_1(a)=\prod_{j=1}^{n-1} a_j^2=B_1(a)^2,
\]
which proves \eqref{eq:B1-sigma1}.
Now let $d>1$ and set $m=n/d$.
Then Lemma~\ref{lem:symmetrized-products-revised} gives
\[
\Sigma_{n/d}(a)=\prod_{\substack{1\le j\le n-1\\ d\nmid j}} a_j^2
=\frac{\prod_{j=1}^{n-1} a_j^2}{\prod_{\substack{1\le j\le n-1\\ d\mid j}} a_j^2}
=\frac{\Sigma_1(a)}{B_d(a)^2}.
\]
Rearranging yields \eqref{eq:Bd-sigma}.
\end{proof}

Putting everything together, the discriminant admits a clean divisor-lattice factorization. This is the form that best reflects the later separation between continuous shape variables and arithmetic growth variables.

\begin{proposition}[Divisor-lattice factorization of the discriminant]
\label{prop:disc-divisor-factorization-revised}
Assume Hypothesis~\textup{(H)}.
Then
\begin{equation}\label{eq:disc-divisor-factorization-revised}
|\disc(K_a)|
=\kappa_n\bigl(S(a)\bigr)\cdot B_1(a)^{\,n-1}
\cdot \prod_{\substack{d\mid n\\ 1<d<n}} B_d(a)^{-\varphi(d)},
\end{equation}
where $\kappa_n\bigl(S(a)\bigr)$ is as in Corollary~\ref{cor:disc-exponent-vector-revised}.
Equivalently,
\[
|\disc(K_a)|
=\kappa_n\bigl(S(a)\bigr)\cdot \bigl(s_1(a)s_{n-1}(a)\bigr)^{\frac{n-1}{2}}
\cdot
\prod_{\substack{d\mid n\\ 1<d<n}}
\left(\frac{s_{n/d}(a)s_{n-n/d}(a)}{s_1(a)s_{n-1}(a)}\right)^{\!\frac{\varphi(d)}{2}}.
\]
\end{proposition}

\begin{proof}
By Corollary~\ref{cor:disc-exponent-vector-revised},
\[
|\disc(K_a)|=\kappa_n\bigl(S(a)\bigr)\prod_{j=1}^{n-1} a_j^{\,n-\gcd(j,n)}.
\]
Fix $j$ and set $g=\gcd(j,n)$.
Using the identity
\[
\sum_{d\mid g}\varphi(d)=g,
\]
we obtain
\[
n-g=(n-1)-\sum_{\substack{d\mid g\\ d>1}}\varphi(d).
\]
Now the exponent of $a_j$ in
\[
B_1(a)^{\,n-1}\cdot \prod_{\substack{d\mid n\\ 1<d<n}} B_d(a)^{-\varphi(d)}
\]
is exactly
\[
(n-1)-\sum_{\substack{d\mid j\\ d>1}}\varphi(d)
=(n-1)-\sum_{\substack{d\mid g\\ d>1}}\varphi(d)
=n-g,
\]
because the divisors $d\mid n$ with $d\mid j$ are precisely the divisors of $g$.
This proves \eqref{eq:disc-divisor-factorization-revised}.
The equivalent formula in terms of the $s_m(a)$ follows from
Corollary~\ref{cor:divisor-products-revised}.
\end{proof}

\begin{remark}\label{rem:disc-divisor-interpretation}
Proposition~\ref{prop:disc-divisor-factorization-revised} shows that, after removing the finite periodic factor
$\kappa_n\bigl(S(a)\bigr)$ supported at primes dividing $n$, the discriminant is controlled by the squarefree
``divisor-products'' $B_d(a)$ indexed by the divisor lattice of $n$.
Equivalently, the symmetrized archimedean parameters $\Sigma_m(a)=s_m(a)s_{n-m}(a)$ canonically package the
strong-decomposition variables according to the divisibility of their exponent in $a$.
\end{remark}

\section{Pure locus in shape space}
\label{sec:paired-rational-leaves}

The aim of this section is to isolate the geometric support of the pure-field
shape locus and the coordinate mechanism that governs it.

There are three points to keep distinct.

First, for admissible \(a\) satisfying Hypothesis~\textup{(H)}, the
factorization theorem of Section~\ref{sec:shapes-pure-fields} produces a
rational upper-triangular matrix
\[
C(a)\in \GL_{n-1}(\Q)
\]
such that
\[
\operatorname{sh}(K_a)
=
\Bigl[
C(a)^{\mathsf T}\diag\bigl(s_1(a),\dots,s_{n-1}(a)\bigr)C(a)
\Bigr].
\]
Its right \(U_{n-1}^+(\Z)\)-class
\[
\Xi(a)=[C(a)]\in \GL_{n-1}(\Q)/U_{n-1}^+(\Z)
\]
is the fine presentation-dependent normalized stratum parameter.

Second, the actual geometric support of the shape is encoded by the rational
diagonal leaf. For a rational matrix \(C\), one considers
\[
\mathcal T_C
:=
\{[C^{\mathsf T}DC]: D \text{ positive diagonal}\}
\subset \mathscr S_{n-1}.
\]
By Lemma~\ref{lem:leaf-equivalence-PCU}, this set depends only on \(C\) up to
positive diagonal rescaling on the left and integral basis change on the
right. We will not use any further identification between distinct normalized
strata.

Third, on a fixed normalized stratum, the shape depends only on ratio variables,
whereas the discriminant, after removing the finite periodic factor, depends
only on product variables.  This is the structural reason that the natural
counting problem for pure-field shapes is leafwise rather than ambient.

\subsection{Paired coordinates and fixed-stratum reduction}

Throughout this section, let \(a\) be admissible and write its strong
decomposition as in \eqref{eq:strong-decomp-disc}:
\[
a=\varepsilon\prod_{j=1}^{n-1} a_j^{\,j}.
\]

Recall from \eqref{eq:sm-def} that
\[
s_m(a)=\frac{|a|^{2m/n}}{C_m(a)^2}
\qquad (1\le m\le n-1),
\]
and hence, by Lemma~\ref{lem:s-strong},
\begin{equation}\label{eq:s-strong-paired}
s_m(a)=\prod_{j=1}^{n-1} a_j^{\,2\{jm/n\}},
\qquad
\{x\}:=x-\lfloor x\rfloor.
\end{equation}

Set
\[
t:=\Bigl\lfloor \frac{n-1}{2}\Bigr\rfloor.
\]
For \(1\le j\le t\), define
\[
u_j:=a_j,\qquad
v_j:=a_{n-j},\qquad
P_j:=u_jv_j,\qquad
\rho_j:=\frac{u_j}{v_j}.
\]
If \(n\) is even, also set
\[
b:=a_{n/2}.
\]
Let
\[
\mathcal J_{\mathrm{cop}}
:=
\{1\le j\le t:\ \gcd(j,n)=1\},
\qquad
d:=|\mathcal J_{\mathrm{cop}}|=\frac{\varphi(n)}{2}.
\]
Finally, define the discrete label
\begin{equation}\label{eq:delta-fixed-leaf}
\delta(a)
:=
\Bigl(
(u_j,v_j)_{\substack{1\le j\le t\\ \gcd(j,n)>1}},
\ b\text{ if }n\text{ is even}
\Bigr).
\end{equation}
Thus \(\delta(a)\) records exactly the strong-decomposition variables attached
to non-coprime indices.

For an abstract discrete label
\[
\delta=
\Bigl(
(u_j,v_j)_{\substack{1\le j\le t\\ \gcd(j,n)>1}},
\ b\text{ if }n\text{ is even}
\Bigr),
\]
define
\begin{equation}\label{eq:snc-fixed-leaf}
s_m^{\mathrm{nc}}(\delta)
:=
\prod_{\substack{1\le j\le t\\ \gcd(j,n)>1}}
u_j^{\,2\{mj/n\}}v_j^{\,2\{m(n-j)/n\}}
\cdot
\begin{cases}
b^{\,2\{m/2\}},& n\text{ even},\\
1,& n\text{ odd},
\end{cases}
\qquad (1\le m\le n-1).
\end{equation}
For \(\rho=(\rho_j)_{j\in \mathcal J_{\mathrm{cop}}}\in (\R_{>0})^d\), set
\begin{equation}\label{eq:D-rho-delta-fixed-leaf}
D(\rho;\delta)
:=
\diag\Bigl(
s_m^{\mathrm{nc}}(\delta)
\prod_{j\in \mathcal J_{\mathrm{cop}}}\rho_j^{\,2\{mj/n\}-1}
\Bigr)_{1\le m\le n-1}.
\end{equation}

We now repackage the diagonal data in paired coordinates \((u_j,v_j)\). In these variables one sees directly that ratios and products play fundamentally different roles.

\begin{proposition}[Diagonal parameters in paired coordinates]
\label{prop:diag-paired-fixed-leaf}
For every admissible \(a\) and every \(1\le m\le n-1\),
\begin{equation}\label{eq:s-m-paired-fixed-leaf}
s_m(a)
=
s_m^{\mathrm{nc}}\bigl(\delta(a)\bigr)
\cdot
\Bigl(\prod_{j\in \mathcal J_{\mathrm{cop}}}P_j\Bigr)
\cdot
\prod_{j\in \mathcal J_{\mathrm{cop}}}\rho_j^{\,2\{mj/n\}-1}.
\end{equation}
In particular, the projective class
\[
\bigl[\diag(s_1(a),\dots,s_{n-1}(a))\bigr]
\in \Sym_{n-1}^+(\R)/\R_{>0}
\]
depends only on \(\delta(a)\) and the coprime ratios
\[
\rho(a):=(\rho_j(a))_{j\in \mathcal J_{\mathrm{cop}}},
\]
and is independent of the coprime products
\[
(P_j)_{j\in \mathcal J_{\mathrm{cop}}}.
\]
\end{proposition}

\begin{proof}
Starting from \eqref{eq:s-strong-paired}, pair the indices \(j\) and \(n-j\)
for \(1\le j\le t\), and if \(n\) is even isolate the middle index \(j=n/2\).

Fix \(1\le j\le t\) and \(1\le m\le n-1\), and put
\[
A:=\Bigl\{\frac{mj}{n}\Bigr\},
\qquad
B:=\Bigl\{\frac{m(n-j)}{n}\Bigr\}.
\]
Then
\[
u_j^{2A}v_j^{2B}
=
(u_jv_j)^{A+B}(u_j/v_j)^{A-B}
=
P_j^{\,A+B}\rho_j^{\,A-B}.
\]
Moreover,
\[
\frac{m(n-j)}{n}=m-\frac{mj}{n},
\]
so \(B=\{-mj/n\}\).  If \(n\nmid mj\), then \(A\in (0,1)\), hence \(B=1-A\),
and therefore
\[
A+B=1,
\qquad
A-B=2A-1=2\{mj/n\}-1.
\]
If \(n\mid mj\), then \(A=B=0\).

Now suppose \(j\in \mathcal J_{\mathrm{cop}}\). Since \(\gcd(j,n)=1\) and
\(1\le m\le n-1\), one has \(n\nmid mj\). Therefore the \(j\)-pair contributes
\[
u_j^{\,2\{mj/n\}}v_j^{\,2\{m(n-j)/n\}}
=
P_j\,\rho_j^{\,2\{mj/n\}-1}.
\]
If \(\gcd(j,n)>1\), keep the pair contribution inside
\(s_m^{\mathrm{nc}}(\delta(a))\).

If \(n\) is even, the middle index contributes
\[
b^{\,2\{m/2\}}
=
\begin{cases}
b,& m\text{ odd},\\
1,& m\text{ even}.
\end{cases}
\]
Combining all factors yields \eqref{eq:s-m-paired-fixed-leaf}.  The final
assertion follows because the factor
\[
\prod_{j\in \mathcal J_{\mathrm{cop}}}P_j
\]
is independent of \(m\) and therefore contributes only a global scalar.
\end{proof}

\subsection{Normalized strata}

The factorization theorem of Section~\ref{sec:shapes-pure-fields} shows that,
for admissible \(a\) satisfying Hypothesis~\textup{(H)}, the class
\[
\Xi(a)=[C(a)]\in \GL_{n-1}(\Q)/U_{n-1}^+(\Z)
\]
is the natural fine discrete parameter attached to the chosen normalized
integral basis. Once a normalized stratum has been fixed, it is natural to look at the full diagonal family it supports inside shape space. The next definition isolates that support set.

\begin{definition}[Support of a normalized stratum]
\label{def:translated-diagonal-stratum}
Let
\[
\Xi\in \GL_{n-1}(\Q)/U_{n-1}^+(\Z).
\]
Choose any representative \(C\in \GL_{n-1}(\Q)\) of \(\Xi\), and define
\[
\mathscr T_{\Xi}
:=
\Bigl\{
[C^{\mathsf T}\diag(x_1,\dots,x_{n-1})C]:
(x_1,\dots,x_{n-1})\in (\R_{>0})^{n-1}
\Bigr\}
\subset \mathscr S_{n-1}.
\]
This subset is independent of the chosen representative \(C\).
\end{definition}

\begin{proof}
If \(C'=CU\) with \(U\in U_{n-1}^+(\Z)\), then
\[
(C')^{\mathsf T}\diag(x_1,\dots,x_{n-1})C'
=
U^{\mathsf T}\bigl(C^{\mathsf T}\diag(x_1,\dots,x_{n-1})C\bigr)U,
\]
which defines the same point of \(\mathscr S_{n-1}\).
\end{proof}

Once a normalized stratum has been fixed, the previous formula shows that the remaining shape variation is carried entirely by the ratio variables. This is the basic reduction behind the later leafwise geometry.

\begin{proposition}[Shape on a fixed normalized stratum]
\label{prop:shape-on-fixed-leaf}
Assume that \(a\) is admissible and satisfies Hypothesis~\textup{(H)}.
Let \(C\in \GL_{n-1}(\Q)\) be any representative of \(\Xi(a)\).
Then
\begin{equation}\label{eq:shape-on-fixed-leaf}
\operatorname{sh}(K_a)
=
\bigl[C^{\mathsf T}D\bigl(\rho(a);\delta(a)\bigr)C\bigr].
\end{equation}
In particular, once \(\Xi(a)\) is fixed, the shape of \(K_a\) is determined by
the pair \(\bigl(\delta(a),\rho(a)\bigr)\).
\end{proposition}

\begin{proof}
By Theorem~\ref{thm:factorization},
\[
\operatorname{sh}(K_a)
=
\bigl[
C(a)^{\mathsf T}\diag(s_1(a),\dots,s_{n-1}(a))C(a)
\bigr].
\]
If \(C\) represents \(\Xi(a)\), then \(C(a)=CU\) for some
\(U\in U_{n-1}^+(\Z)\), so
\[
\bigl[
C(a)^{\mathsf T}\diag(s_m(a))C(a)
\bigr]
=
\bigl[
C^{\mathsf T}\diag(s_m(a))C
\bigr]
\]
in \(\mathscr S_{n-1}\).

Now apply Proposition~\ref{prop:diag-paired-fixed-leaf}:
\[
\diag(s_1(a),\dots,s_{n-1}(a))
=
\Bigl(\prod_{j\in \mathcal J_{\mathrm{cop}}}P_j\Bigr)
D\bigl(\rho(a);\delta(a)\bigr).
\]
The prefactor is a positive scalar and is therefore invisible in shape space.
\end{proof}

The discriminant behaves dually: after removing the periodic finite factor, the coprime pairs contribute only through their products, not through their ratios. This is the arithmetic counterpart of the previous shape statement.

\begin{proposition}
\label{prop:disc-pairs-only-products-fixed-leaf}
Assume that \(a\) is admissible and satisfies Hypothesis~\textup{(H)}.
Write
\[
\kappa_n(a):=\kappa_n\bigl(S(a)\bigr)
\]
for the periodic factor of Corollary~\ref{cor:disc-exponent-vector-revised}, and
let
\[
w_j:=n-\gcd(j,n)
\qquad (1\le j\le n-1).
\]
Then
\begin{equation}\label{eq:disc-paired-clean-fixed-leaf}
|\disc(K_a)|
=
\kappa_n(a)
\cdot
\Bigl(\prod_{j\in \mathcal J_{\mathrm{cop}}}P_j^{\,n-1}\Bigr)
\cdot
G\bigl(\delta(a)\bigr),
\end{equation}
where
\[
G\bigl(\delta(a)\bigr)
:=
\Bigl(
\prod_{\substack{1\le j\le t\\ \gcd(j,n)>1}}
(u_jv_j)^{w_j}
\Bigr)
\cdot
\begin{cases}
b^{w_{n/2}},& n\text{ even},\\
1,& n\text{ odd}.
\end{cases}
\]
Equivalently,
\begin{equation}\label{eq:disc-paired-normalized-fixed-leaf}
\frac{|\disc(K_a)|}{\kappa_n(a)}
=
\Bigl(\prod_{j\in \mathcal J_{\mathrm{cop}}}P_j^{\,n-1}\Bigr)
\cdot
G\bigl(\delta(a)\bigr).
\end{equation}
In particular, after fixing \(\kappa_n(a)\), the discriminant depends on the
coprime pairs only through the products \(P_j\), and not through the ratios
\(\rho_j\).
\end{proposition}

\begin{proof}
By Corollary~\ref{cor:disc-exponent-vector-revised},
\[
|\disc(K_a)|=\kappa_n(a)\prod_{j=1}^{n-1} a_j^{\,w_j}.
\]
Pair the indices \(j\) and \(n-j\) for \(1\le j\le t\).

If \(j\in \mathcal J_{\mathrm{cop}}\), then \(\gcd(j,n)=1\), so
\[
w_j=w_{n-j}=n-1,
\]
and therefore
\[
a_j^{w_j}a_{n-j}^{w_{n-j}}
=
u_j^{n-1}v_j^{n-1}
=
P_j^{n-1}.
\]
If \(\gcd(j,n)>1\), keep the pair contribution \((u_jv_j)^{w_j}\) inside
\(G(\delta(a))\).  If \(n\) is even, the middle index contributes
\(b^{w_{n/2}}\).  Combining these factors gives
\eqref{eq:disc-paired-clean-fixed-leaf}, and
\eqref{eq:disc-paired-normalized-fixed-leaf} is immediate.
\end{proof}

So on a fixed stratum the discriminant cutoff becomes geometrically simple: it is just a hyperbolic bound in the product variables.

\begin{corollary}
\label{cor:disc-product-constraint-fixed-leaf}
Fix a real number \(\kappa>0\) and a discrete label \(\delta\).
For admissible parameters \(a\) satisfying Hypothesis~\textup{(H)} and
\[
\kappa_n(a)=\kappa,
\qquad
\delta(a)=\delta,
\]
the condition \(|\disc(K_a)|\le X\) is equivalent to
\[
\prod_{j\in \mathcal J_{\mathrm{cop}}}P_j
\le
\left(\frac{X}{\kappa\,G(\delta)}\right)^{1/(n-1)}.
\]
\end{corollary}

\begin{proof}
This is immediate from Proposition~\ref{prop:disc-pairs-only-products-fixed-leaf}.
\end{proof}

For later reference, if \(\Xi\in \GL_{n-1}(\Q)/U_{n-1}^+(\Z)\) is fixed with
representative \(C\) and \(\delta\) is a discrete label, define
\begin{equation}\label{eq:F-Xi-delta-lite}
\mathcal F_{\Xi,\delta}
:=
\bigl\{
[C^{\mathsf T}D(\rho;\delta)C]:
\rho\in (\R_{>0})^d
\bigr\}
\subset \mathscr T_{\Xi}.
\end{equation}
Then Proposition~\ref{prop:shape-on-fixed-leaf} shows that every admissible
\(a\) satisfying Hypothesis~\textup{(H)} with
\[
\Xi(a)=\Xi,
\qquad
\delta(a)=\delta
\]
satisfies
\[
\operatorname{sh}(K_a)\in \mathcal F_{\Xi,\delta}.
\]

\subsection{Geometry of rational diagonal leaves}

For \(m\ge 2\) and \(C\in \GL_m(\Q)\), define
\[
\Phi_C:(\R_{>0})^m/\R_{>0}\longrightarrow \mathscr S_m,
\qquad
[D]\longmapsto [C^{\mathsf T}DC],
\]
where \(D\) is positive diagonal, and set
\begin{equation}\label{eq:type-locus-fixed-C}
\mathcal T_C
:=
\operatorname{im}(\Phi_C)
=
\{[C^{\mathsf T}DC]: D\text{ positive diagonal}\}
\subset \mathscr S_m.
\end{equation}

Before studying these leaves geometrically, it is useful to record the elementary operations that do not change them. This is the basic equivalence relation underlying all later normalizations.

\begin{lemma}
\label{lem:leaf-equivalence-PCU}
Let \(m\ge 2\), let \(C\in \GL_m(\Q)\), let \(P\) be a positive diagonal matrix
in \(\GL_m(\R)\), and let \(U\in \GL_m(\Z)\). Then
\[
\mathcal T_{PCU}=\mathcal T_C.
\]
\end{lemma}

\begin{proof}
For any positive diagonal matrix \(D\),
\[
(PCU)^{\mathsf T}D(PCU)
=
U^{\mathsf T}C^{\mathsf T}(P^{\mathsf T}DP)CU.
\]
Since \(P\) is positive diagonal, the map \(D\mapsto P^{\mathsf T}DP\) is a
bijection on the set of positive diagonal matrices.  Since
\(U\in \GL_m(\Z)\), the outer factor \(U^{\mathsf T}(\cdot)U\) does not change
the corresponding point of \(\mathscr S_m\).
\end{proof}

The next point is a compactness statement: a rational diagonal leaf cannot wander through a compact part of shape space while its diagonal ratios degenerate. This is the key input behind closedness.

\begin{lemma}[Properness of a rational diagonal translate]
\label{lem:proper-rational-leaf}
Let \(m\ge 2\) and \(C\in \GL_m(\Q)\). Then \(\Phi_C\) is proper.
Equivalently, for every compact subset \(K\subset \mathscr S_m\) there exists
\(M_K\ge 1\) such that
\[
[C^{\mathsf T}DC]\in K
\quad\Longrightarrow\quad
M_K^{-1}\le \frac{s_i}{s_j}\le M_K
\qquad (1\le i,j\le m),
\]
for every positive diagonal matrix \(D=\diag(s_1,\dots,s_m)\).
\end{lemma}

\begin{proof}
Represent each class \([D]\) by the unique diagonal matrix with \(\det(D)=1\).
Suppose, for contradiction, that there exists a compact set
\(K\subset \mathscr S_m\) and a sequence
\[
D_k=\diag(s_{1,k},\dots,s_{m,k}),
\qquad
\det(D_k)=1,
\]
such that
\[
[C^{\mathsf T}D_kC]\in K
\qquad\text{for all }k,
\]
but the ratios \(s_{i,k}/s_{j,k}\) are unbounded.  After passing to a
subsequence, there is an index \(i_0\) such that
\[
s_{i_0,k}\longrightarrow 0.
\]

Choose an integer \(q\ge 1\) such that
\[
A:=qC\in M_m(\Z).
\]
Let \(w\in \Z^m\) be the \(i_0\)-th column of \(\operatorname{adj}(A)\).  Then
\(w\neq 0\) and
\[
Aw=\det(A)e_{i_0}.
\]
Therefore
\[
Cw=\frac{\det(A)}{q}e_{i_0}.
\]
Set
\[
G_k:=C^{\mathsf T}D_kC.
\]
Then the squared length of the nonzero integral vector \(w\) with respect to
\(G_k\) is
\[
w^{\mathsf T}G_kw
=
w^{\mathsf T}C^{\mathsf T}D_kCw
=
(Cw)^{\mathsf T}D_k(Cw)
=
\left(\frac{\det(A)}{q}\right)^2 s_{i_0,k}
\longrightarrow 0.
\]
On the other hand,
\[
\det(G_k)=\det(C)^2\det(D_k)=\det(C)^2
\]
is independent of \(k\).

After rescaling each \(G_k\) by the fixed scalar \(|\det(C)|^{-2/m}\), we
obtain determinant-one positive-definite quadratic forms representing the same
points of shape space.  Their shortest nonzero vector lengths tend to \(0\).
By Mahler's compactness criterion, their classes leave every compact subset of
\(\mathscr S_m\), contradicting \([G_k]\in K\).  This proves the asserted ratio
bound.

Now let \(K\subset \mathscr S_m\) be compact.  The ratio bound shows that
\(\Phi_C^{-1}(K)\) is contained in a compact subset of
\((\R_{>0})^m/\R_{>0}\): after normalizing by \(\det(D)=1\), all diagonal
entries lie in a common compact interval.  Since \(\Phi_C^{-1}(K)\) is closed,
it is compact.  Hence \(\Phi_C\) is proper.
\end{proof}

Properness has an immediate topological consequence: each rational diagonal leaf is actually a closed subset of shape space.

\begin{corollary}[Closedness of rational diagonal leaves]
\label{cor:closed-rational-leaf}
For every \(m\ge 2\) and every \(C\in \GL_m(\Q)\), the subset
\(\mathcal T_C\subset \mathscr S_m\) is closed.
\end{corollary}

\begin{proof}
A proper map between locally compact Hausdorff spaces is closed.  Since
\(\mathcal T_C=\operatorname{im}(\Phi_C)\), the claim follows from
Lemma~\ref{lem:proper-rational-leaf}.
\end{proof}

The previous topological statements become more conceptual once one identifies these leaves with projected flats attached to \(\Q\)-split tori. The next theorem places the whole picture inside the symmetric-space geometry of \(\GL_m\).

\begin{theorem}[Rational diagonal leaves are projected rational flats]
\label{thm:rational-leaf-projected-flat}
Let \(m\ge 2\), let
\[
\Gamma_m:=\GL_m(\Z),\qquad
G_m:=\GL_m(\R),\qquad
K_m:=\GO_m(\R),\qquad
X_m:=G_m/K_m,
\]
and let \(x_0:=K_m\in X_m\).  Let
\[
A_m^1
:=
\{\diag(t_1,\dots,t_m): t_i>0,\ \prod_{i=1}^m t_i=1\}
\subset G_m.
\]
For \(C\in \GL_m(\Q)\), set \(g:=C^{\mathsf T}\) and
\[
F_C:=gA_m^1x_0\subset X_m.
\]
Then:
\begin{enumerate}[label=\textup{(\roman*)},leftmargin=2.8em]
\item Under the identification
\[
\mathscr S_m\simeq \Gamma_m\backslash X_m,
\qquad
hK_m\longmapsto [hh^{\mathsf T}],
\]
one has
\[
\mathcal T_C=\pi(F_C),
\]
where \(\pi:X_m\to \Gamma_m\backslash X_m\) is the quotient map.

\item
The torus
\[
S_C:=gA_m^1g^{-1}
\]
is a maximal \(\Q\)-split torus of \(\mathrm{SL}_{m,\Q}\).  Equivalently, its
image in \(\mathrm{PGL}_{m,\Q}\) is a maximal \(\Q\)-split torus of
\(\mathrm{PGL}_{m,\Q}\).  Moreover
\[
F_C=S_C\cdot (g x_0),
\]
so \(F_C\) is a rational maximal flat in \(X_m\).

\item If
\[
\Gamma_{F_C}:=\{\gamma\in \Gamma_m:\ \gamma F_C=F_C\},
\]
then \(\Gamma_{F_C}\) acts properly discontinuously on \(F_C\), and
\(\pi|_{F_C}\) factors through a real-analytic totally geodesic map
\[
\bar\iota_C:\Gamma_{F_C}\backslash F_C\longrightarrow \mathscr S_m
\]
whose image is exactly \(\mathcal T_C\). In particular, \(\mathcal T_C\) is
closed.
\end{enumerate}
\end{theorem}

\begin{proof}
Under the identification
\(X_m\simeq \Sym_m^+(\R)/\R_{>0}\), a point \(hK_m\) corresponds to
\([hh^{\mathsf T}]\). Thus for \(a\in A_m^1\),
\[
\pi(ga x_0)
=
[ga^2g^{\mathsf T}]
=
[C^{\mathsf T}a^2C].
\]
As \(a\) ranges over \(A_m^1\), the matrix \(a^2\) ranges over the positive
diagonal matrices of determinant \(1\), whose projective classes are precisely
the projective classes of all positive diagonal matrices. Hence
\(\pi(F_C)=\mathcal T_C\), proving \textup{(i)}.

For \textup{(ii)}, \(A_m^1\) is the standard diagonal torus of
\(\mathrm{SL}_{m,\Q}\), hence a maximal \(\Q\)-split torus.  Since
\(g\in\GL_m(\Q)\), conjugation by \(g\) is a \(\Q\)-automorphism of
\(\mathrm{SL}_{m,\Q}\), so
\[
S_C:=gA_m^1g^{-1}
\]
is again a maximal \(\Q\)-split torus.  Passing to the adjoint quotient gives
the corresponding statement for \(\mathrm{PGL}_{m,\Q}\).  Finally,
\[
S_C\cdot(gx_0)=gA_m^1x_0=F_C.
\]
Since
\[
X_m=\GL_m(\R)/\GO_m(\R)
\simeq
\mathrm{PGL}_m(\R)/\mathrm{PO}_m(\R),
\]
the orbit of a maximal \(\R\)-split torus through any point of \(X_m\) is a
maximal flat.  Thus \(F_C\) is a rational maximal flat.

For \textup{(iii)}, since \(\Gamma_m\) acts properly discontinuously on \(X_m\),
the subgroup \(\Gamma_{F_C}\) acts properly discontinuously on the closed
totally geodesic flat \(F_C\). The restriction \(\pi|_{F_C}\) is
\(\Gamma_{F_C}\)-invariant, so it descends to a real-analytic map
\[
\bar\iota_C:\Gamma_{F_C}\backslash F_C\longrightarrow \mathscr S_m.
\]
Because \(F_C\) is totally geodesic in \(X_m\), this descended map is totally
geodesic, and its image is
\[
\bar\iota_C(\Gamma_{F_C}\backslash F_C)=\pi(F_C)=\mathcal T_C
\]
by \textup{(i)}. Closedness was already proved in
Corollary~\ref{cor:closed-rational-leaf}.
\end{proof}

\begin{remark}[Why the full stabilizer is necessary]
\label{rem:full-stabilizer-necessary}
The quotient in
Theorem~\ref{thm:rational-leaf-projected-flat} must be taken by the full
stabilizer of the flat, not merely by its translation subgroup.  For example,
when \(m=2\) and \(C=I_2\), the permutation matrix
\[
\begin{pmatrix}0&1\\ 1&0\end{pmatrix}\in \GL_2(\Z)
\]
preserves the standard diagonal flat and identifies \(a\) with \(a^{-1}\),
although there is no nontrivial translation lattice in that rank-one example.
\end{remark}

Even without Hypothesis~\textup{(H)}, the pure locus is already leafwise rather than ambient. The next proposition is the unconditional support statement.

\begin{proposition}[Unconditional rational-leaf support]
\label{prop:ambient-rational-leaf-support}
Let
\[
\mathcal S_{n,\mathrm{adm}}^{\mathrm{pure}}
:=
\{\operatorname{sh}(K_a): a\text{ admissible}\}
\subset \mathscr S_{n-1},
\]
and
\[
\mathcal L_n^{\mathrm{rat}}
:=
\bigcup_{C\in \GL_{n-1}(\Q)}\mathcal T_C
\subset \mathscr S_{n-1}.
\]
Then
\[
\mathcal S_{n,\mathrm{adm}}^{\mathrm{pure}}
\subset
\mathcal L_n^{\mathrm{rat}}.
\]
Consequently, the admissible pure-field shape locus is contained in a countable
union of closed rational diagonal leaves.
\end{proposition}

\begin{proof}
Fix an admissible parameter \(a\), and choose a \(\Z\)-basis
\(\eta_1,\dots,\eta_{n-1}\) of \(\cO_{K_a}^{\perp}\).
By Lemma~\ref{lem:all-gram-factorization}, there exists
\[
A(a)\in \GL_{n-1}(\Q)
\]
such that the Gram matrix of
\[
J(\eta_1),\dots,J(\eta_{n-1})
\]
is
\[
A(a)^{\mathsf T}\bigl(n\Delta(a)\bigr)A(a),
\qquad
\Delta(a):=\diag\bigl(s_1(a),\dots,s_{n-1}(a)\bigr).
\]
Since the scalar factor \(n\) is irrelevant in shape space, this shows that
\[
\operatorname{sh}(K_a)\in \mathcal T_{A(a)}.
\]
Hence
\[
\mathcal S_{n,\mathrm{adm}}^{\mathrm{pure}}
\subset
\mathcal L_n^{\mathrm{rat}}.
\]
Because \(\GL_{n-1}(\Q)\) is countable, \(\mathcal L_n^{\mathrm{rat}}\) is a
countable union of the leaves \(\mathcal T_C\), and each \(\mathcal T_C\) is
closed by Corollary~\ref{cor:closed-rational-leaf}.
\end{proof}

Recall that a full-rank matrix \(H=(h_{ij})\in M_m(\Z)\) is in
\emph{right Hermite normal form} if:
\begin{enumerate}[label=\textup{(\roman*)},leftmargin=2.6em]
\item \(H\) is upper triangular;
\item \(h_{ii}>0\) for all \(i\);
\item \(0\le h_{ij}<h_{ii}\) for all \(i<j\).
\end{enumerate}
Every full-rank integer matrix \(A\in M_m(\Z)\) admits a unique
\(U\in \GL_m(\Z)\) such that \(AU\) is in right Hermite normal form.

This elementary finiteness lemma is the combinatorial reason that a uniform denominator bound translates into only finitely many leaf types.

\begin{lemma}[Finiteness of bounded-determinant right Hermite normal forms]
\label{lem:finite-HNF-bounded-det}
Fix integers \(m\ge 1\) and \(B\ge 1\). Then there are only finitely many
full-rank right Hermite normal form matrices \(H\in M_m(\Z)\) satisfying
\[
\det(H)\le B.
\]
\end{lemma}

\begin{proof}
Since \(H\) is upper triangular with positive diagonal,
\[
\det(H)=\prod_{i=1}^m h_{ii}.
\]
Thus the diagonal entries form an \(m\)-tuple of positive integers with product
at most \(B\), and there are only finitely many such tuples.

Once the diagonal is fixed, each off-diagonal entry \(h_{ij}\) with \(i<j\)
belongs to the finite set \(\{0,1,\dots,h_{ii}-1\}\). Hence there are only
finitely many such matrices.
\end{proof}

We can now strengthen countable rational-leaf support to finite leaf support under Hypothesis~\textup{(H)}. The key inputs are the uniform denominator bound and reduction to Hermite normal form.

\begin{theorem}[Finite rational-leaf containment under Hypothesis~\textup{(H)}]
\label{thm:countable-thinness-fixed-leaf}
Assume Hypothesis~\textup{(H)} and define
\[
\mathcal S_n^{\mathrm{pure},H}
:=
\{\operatorname{sh}(K_a): a\text{ admissible and satisfying \textup{(H)}}\}
\subset \mathscr S_{n-1}.
\]
Then there exists a finite set
\[
\mathcal H_n\subset M_{n-1}(\Z)
\]
of full-rank right Hermite normal form matrices, depending only on \(n\), such
that
\[
\mathcal S_n^{\mathrm{pure},H}
\subset
\bigcup_{H\in \mathcal H_n}\mathcal T_H.
\]
Moreover, each \(\mathcal T_H\) is closed in \(\mathscr S_{n-1}\), and is the
image of a real-analytic map from an \((n-2)\)-dimensional manifold.
\end{theorem}

\begin{proof}
Put \(m:=n-1\) and
\[
N_n^\sharp:=\prod_{p^e\parallel n} p^{\,e+n-2}.
\]
Let \(a\) be admissible and satisfy Hypothesis~\textup{(H)}.
By Proposition~\ref{prop:uniform-denominator},
\[
R(a):=N_n^\sharp C(a)\in M_m(\Z).
\]
Since \(N_n^\sharp I_m\) is a positive diagonal matrix,
Lemma~\ref{lem:leaf-equivalence-PCU} gives
\[
\mathcal T_{R(a)}=\mathcal T_{C(a)}.
\]

Choose \(U(a)\in \GL_m(\Z)\) such that
\[
H(a):=R(a)U(a)
\]
is in right Hermite normal form. Again by
Lemma~\ref{lem:leaf-equivalence-PCU},
\[
\mathcal T_{H(a)}=\mathcal T_{R(a)}=\mathcal T_{C(a)}.
\]

Moreover,
\[
\det H(a)=\det R(a)=(N_n^\sharp)^m\det C(a).
\]
By Lemma~\ref{lem:C-matrix}\textup{(b)},
\[
\det C(a)=\prod_{j=1}^m D_j(a)^{-1}.
\]
Since each \(D_j(a)\) is a positive integer, one has
\[
0<\det C(a)\le 1.
\]
Hence
\[
1\le \det H(a)\le (N_n^\sharp)^m.
\]
Lemma~\ref{lem:finite-HNF-bounded-det} therefore shows that only finitely many
such Hermite normal forms can occur. Let \(\mathcal H_n\) be the resulting
finite set. Then
\[
\mathcal T_{C(a)}=\mathcal T_{H(a)}
\qquad\text{for some }H(a)\in \mathcal H_n,
\]
which proves the containment.

Finally, closedness follows from Corollary~\ref{cor:closed-rational-leaf}, and
the real-analytic parametrization is
\[
(\R_{>0})^m/\R_{>0}\longrightarrow \mathscr S_m,
\qquad
[D]\longmapsto [H^{\mathsf T}DH],
\]
whose domain has dimension \(m-1=n-2\).
\end{proof}

The sextic example separates normalized strata, but it does not yet separate the underlying rational leaves. The following remark clarifies exactly what the current argument proves and what it leaves open.

\begin{remark}
\label{rem:sextic-same-leaf}
Example~\ref{prop:counterexample-C} proves that
\[
S(10)=S(550)
\qquad\text{but}\qquad
\Xi(10)\neq \Xi(550).
\]
Thus the periodic datum \(S(a)\) does not determine the normalized
stratum \(\Xi(a)\).

It does \emph{not} follow from the present argument that the corresponding
rational diagonal leaves are equal or different. Indeed, after clearing
denominators by \(3\),
\[
3C(10)=
\begin{pmatrix}
3&0&0&0&1\\
0&3&0&10&0\\
0&0&3&0&10\\
0&0&0&1&0\\
0&0&0&0&1
\end{pmatrix},
\qquad
3C(550)=
\begin{pmatrix}
3&0&0&0&5\\
0&3&0&2750&0\\
0&0&3&0&13750\\
0&0&0&1&0\\
0&0&0&0&1
\end{pmatrix}.
\]
The unique right Hermite normal forms of these two matrices are
\[
H_{10}=
\begin{pmatrix}
3&0&0&0&1\\
0&3&0&1&0\\
0&0&3&0&1\\
0&0&0&1&0\\
0&0&0&0&1
\end{pmatrix},
\qquad
H_{550}=
\begin{pmatrix}
3&0&0&0&2\\
0&3&0&2&0\\
0&0&3&0&1\\
0&0&0&1&0\\
0&0&0&0&1
\end{pmatrix},
\]
which are different. Therefore the previous right-Hermite-normal-form argument
does not prove
\[
\mathcal T_{C(10)}=\mathcal T_{C(550)}.
\]
Accordingly, in this paper we use Example~\ref{prop:counterexample-C} only for
the statement
\[
S(a)\not\Rightarrow \Xi(a).
\]
Whether \(\mathcal T_{C(10)}\) and \(\mathcal T_{C(550)}\) coincide is left
open here.
\end{remark}

At this point the counting picture becomes transparent. On a fixed normalized stratum, shape motion and discriminant growth are governed by disjoint sets of variables.

\begin{proposition}[Fixed-stratum reduction to ratios and products]
\label{prop:fixed-leaf-counting-mechanism}
Assume Hypothesis~\textup{(H)}.
Fix a normalized stratum
\[
\Xi\in \GL_{n-1}(\Q)/U_{n-1}^+(\Z),
\]
choose a representative \(C\in \GL_{n-1}(\Q)\) of \(\Xi\), fix a discrete
label \(\delta\) as in \eqref{eq:delta-fixed-leaf}, and fix a value
\(\kappa>0\) of the periodic discriminant factor \(\kappa_n(a)\).
For every admissible parameter \(a\) satisfying
\[
\Xi(a)=\Xi,\qquad
\delta(a)=\delta,\qquad
\kappa_n(a)=\kappa,
\]
one has
\begin{equation}\label{eq:fixed-leaf-shape-ratio}
\operatorname{sh}(K_a)
=
\bigl[C^{\mathsf T}D\bigl(\rho(a);\delta\bigr)C\bigr]
\in \mathcal F_{\Xi,\delta},
\end{equation}
and
\begin{equation}\label{eq:fixed-leaf-disc-product}
|\disc(K_a)|\le X
\iff
\prod_{j\in \mathcal J_{\mathrm{cop}}}P_j(a)
\le
\left(\frac{X}{\kappa\,G(\delta)}\right)^{1/(n-1)}.
\end{equation}
Thus, on a fixed normalized stratum and after fixing the non-coprime discrete
data, the shape depends only on the ratio variables \(\rho(a)\), whereas the
discriminant bound depends only on the product variables \(P_j(a)\).
\end{proposition}

\begin{proof}
Equation \eqref{eq:fixed-leaf-shape-ratio} is
Proposition~\ref{prop:shape-on-fixed-leaf}, and
\eqref{eq:fixed-leaf-disc-product} is
Corollary~\ref{cor:disc-product-constraint-fixed-leaf}.
\end{proof}

\begin{remark}
\label{rem:low-degree-leafwise-mechanism}
Proposition~\ref{prop:fixed-leaf-counting-mechanism} isolates the common
mechanism behind the currently known distribution results for pure fields.
On a fixed normalized stratum, the archimedean variation is carried by the
ratio variables, the discriminant contributes only the hyperbolic product
constraint, and the remaining arithmetic survives as finitely many discrete
labels and congruence conditions.  In low degrees the explicit arithmetic types
are sharp enough that these strata coincide with the torus-orbit pieces on
which the limiting measures live.
\end{remark}

\subsection{Further directions}\label{sec:outlook-open-problems}

The results of the present paper stop at the level of structure and support.
Three natural problems remain.
\begin{enumerate}
    \item \emph{Removing Hypothesis~\textup{(H)}.}
Under Hypothesis~\textup{(H)} we have normalized integral bases, the
factorization through \(C(a)\), and finite rational-leaf support.  The first
next step is to absorb the bad primes \(p\mid n\) with \(p\mid v_p(a)\) into
refined finite-place data defined for every admissible parameter.  One expects
an unconditional factorization
\[
\sh(K_a)=\bigl[(C^\sharp(a))^{\mathsf T}D^\sharp(a)C^\sharp(a)\bigr]
\]
together with the same paired-coordinate ratio/product separation as in the
present paper.  Proving this would upgrade the unconditional countable-leaf
support theorem to unconditional finite leaf support in every fixed degree.
\item \emph{Counting and equidistribution on fixed leaves.}
Section~\ref{sec:paired-rational-leaves} shows that on a fixed normalized
stratum, after the non-coprime discrete data are fixed, the shape depends only
on ratio variables whereas discriminant bounds depend only on independent
product variables.  This identifies the natural counting problem: one should
count admissible parameters on each realized leaf, or on each fixed-stratum
slice inside that leaf, with product variables subject to the hyperbolic
discriminant constraint.  The missing theorem is a leafwise or stratumwise
equidistribution statement, with explicit weights coming from finite local
densities.  The present paper provides the geometric support on which such a
counting theorem would live; it does not yet supply the counting theorem
itself.
\item \emph{Intrinsic reformulation.}
The data \(S(a)\) and \(\Xi(a)\) are attached to a chosen pure presentation
\(K_a=\Q(\theta)\).  The actual support is coarser: different normalized strata
may define the same rational leaf, and different pure presentations of the same
field should not be expected to produce identical auxiliary data.  An intrinsic
next step is therefore to reformulate the supporting leaf, the realized support
piece, and the finite local discriminant factor in a presentation-free language
depending only on the abstract field.  In odd degree this should recover
shape-completeness in intrinsic terms; in even degree it should clarify what
finite auxiliary data must be added to shape to recover the field.
\end{enumerate}

\bibliographystyle{alpha}
\bibliography{References}

@phdthesis{Terr1997,
  author = {Terr, David Charles},
  title = {The Distribution of Shapes of Cubic Orders},
  school = {University of California, Berkeley},
  year = {1997},
  type = {PhD thesis}
}

@article{BhargavaHarron2016,
  author = {Bhargava, Manjul and Harron, Piper},
  title = {The equidistribution of lattice shapes of rings of integers in cubic, quartic, and quintic number fields},
  journal = {Compositio Mathematica},
  volume = {152},
  number = {6},
  pages = {1111--1120},
  year = {2016},
  doi = {10.1112/S0010437X16007260}
}

@article{Hough2019,
  author = {Hough, Robert},
  title = {The shape of cubic fields},
  journal = {Research in the Mathematical Sciences},
  volume = {6},
  year = {2019},
  article-number = {23},
  doi = {10.1007/s40687-019-0185-1}
}

@article{Harron2017PureCubic,
  author = {Harron, Robert},
  title = {The shapes of pure cubic fields},
  journal = {Proceedings of the American Mathematical Society},
  volume = {145},
  number = {2},
  pages = {509--524},
  year = {2017},
  doi = {10.1090/proc/13309}
}

@article{Holmes2025PurePrime,
  author = {Holmes, Erik},
  title = {On the shapes of pure prime-degree number fields},
  journal = {Journal de th\'eorie des nombres de Bordeaux},
  volume = {37},
  number = {1},
  pages = {1--48},
  year = {2025},
  doi = {10.5802/jtnb.1311}
}

@article{PiperHHarron2020,
  author = {{Piper H} and Harron, Robert},
  title = {The shapes of {G}alois quartic fields},
  journal = {Transactions of the American Mathematical Society},
  volume = {373},
  number = {10},
  pages = {7109--7152},
  year = {2020},
  doi = {10.1090/tran/8137}
}

@misc{DasKalaMukhopadhyayRay2025,
  author = {Das, Sudipa and Kala, Sushant and Mukhopadhyay, Arunabha and Ray, Anwesh},
  title = {On the distribution of shapes of pure quartic number fields},
  year = {2025},
  eprint = {2506.23766},
  archivePrefix = {arXiv},
  primaryClass = {math.NT}
}

@misc{JakharKalwaniyaRayRoy2026,
  author = {Jakhar, Anuj and Kalwaniya, Ravi and Ray, Anwesh and Roy, Bidisha},
  title = {On the distribution of shapes of sextic pure number fields},
  year = {2026},
  eprint = {2601.09411},
  archivePrefix = {arXiv},
  primaryClass = {math.NT}
}

@article{JakharKhandujaSangwan2021,
  author = {Jakhar, Anuj and Khanduja, Sudesh K. and Sangwan, Neeraj},
  title = {On integral basis of pure number fields},
  journal = {Mathematika},
  volume = {67},
  number = {1},
  pages = {187--195},
  year = {2021},
  doi = {10.1112/mtk.12067}
}

@misc{NguyenDang2025MinPeriod,
  author = {Nguyen-Dang, Khai-Hoan},
  title = {The minimal periodicity for integral bases of pure number fields},
  year = {2025},
  eprint = {2509.09457},
  archivePrefix = {arXiv},
  primaryClass = {math.NT},
  note = {Accepted at Research in Number Theory}
}

@misc{NguyenDangHung2025Alpha,
  author = {Nguyen-Dang, Khai-Hoan and Hung, Nguyen Thai},
  title = {$\alpha$-monogeneity of pure number fields: criterion and density},
  year = {2025},
  eprint = {2510.20232},
  archivePrefix = {arXiv},
  primaryClass = {math.NT}
}

@article{MantillaSolerMonsurro2016,
  author = {Mantilla-Soler, Guillermo and Monsurr\`o, Marina},
  title = {The shape of {$\mathbb{Z}/\ell\mathbb{Z}$}-number fields},
  journal = {The Ramanujan Journal},
  volume = {39},
  number = {3},
  pages = {451--463},
  year = {2016},
  doi = {10.1007/s11139-015-9744-2}
}

@article{BolanosMantillaSoler2021,
  author = {Bola\~nos, Wilmar and Mantilla-Soler, Guillermo},
  title = {The trace form over cyclic number fields},
  journal = {Canadian Journal of Mathematics},
  volume = {73},
  number = {4},
  pages = {947--969},
  year = {2021},
  doi = {10.4153/S0008414X20000255}
}

@article{BolanosMantillaSoler2023,
  author = {Bola\~nos, Wilmar and Mantilla-Soler, Guillermo},
  title = {The shape of cyclic number fields},
  journal = {Canadian Mathematical Bulletin},
  volume = {66},
  number = {2},
  pages = {599--609},
  year = {2023},
  doi = {10.4153/S0008439522000546}
}

@article{GayVelez1981,
  author  = {Gay, David Andrew and Velez, William Yslas},
  title   = {The torsion group of a radical extension},
  journal = {Pacific Journal of Mathematics},
  volume  = {92},
  number  = {2},
  pages   = {317--331},
  year    = {1981}
}

\end{document}